\documentclass[a4paper]{amsart}

\usepackage{lipsum}

\newtheorem{theorem}{Theorem}[section]
\newtheorem{lemma}[theorem]{Lemma}
\newtheorem{proposition}[theorem]{Proposition}

\newtheorem{conjecture}[theorem]{Conjecture}
\newtheorem{question}[theorem]{Question}

\newtheorem{letterthm}{Theorem}

\newtheorem{lettercor}[letterthm]{Corollary}

\newtheorem{letterthmbody}{Theorem}

\newtheorem{lettercorbody}[letterthmbody]{Corollary}

\theoremstyle{definition}

\theoremstyle{remark}
\newtheorem{remark}[theorem]{Remark}

\usepackage[hyperfootnotes=false,colorlinks,
linkcolor={blue!80!black},
citecolor={red!60!black},
urlcolor={blue!80!black}]{hyperref}

\usepackage[hang,flushmargin]{footmisc}

\usepackage[dvipsnames]{xcolor}
\usepackage{graphicx}
\usepackage[a4paper,top=28mm,bottom=32mm,left=28mm,right=28mm]{geometry}

\usepackage{amsthm,amsmath,amsfonts,amssymb}
\usepackage{mathtools}

\usepackage{textpos}

\usepackage{caption,subcaption}
\usepackage{enumitem}
\usepackage{orcidlink}

\usepackage{tikz, tikz-cd}
\usetikzlibrary{calc}

\usepackage{float}

\numberwithin{equation}{section}

\DeclareMathOperator{\CL}{CL}
\DeclareMathOperator{\ad}{ad}

\makeatletter
\renewcommand{\l@subsection}{%
	\@tocline{2}{0pt}{3pc}{4pc}{\small\normalfont}%
}
\makeatother

\title{Conjugator length in nilpotent groups}

\author[J. Deré]{Jonas Deré\ \orcidlink{0000-0003-0758-9754}}
\address{KU Leuven, Etienne Sabbelaan 53, Kortrijk, Belgium.}
\email{jonas.dere@kuleuven.be}

\author[K. Vandermeersch]{Ken Vandermeersch\ \orcidlink{0009-0006-8360-2841}}
\address{KU Leuven, Etienne Sabbelaan 53, Kortrijk, Belgium.}
\email{ken.vandermeersch@kuleuven.be}
\thanks{K.V.\ is supported by Methusalem grant METH/21/03 --- long-term structural funding of the Flemish government.}


\begin{document}

\begin{abstract} 
For every rational number $\alpha\geq 2$, we construct a 2-step nilpotent group with conjugator length function $\operatorname{CL}(n)\simeq n^\alpha$. We deduce that the nilpotent conjugator length spectrum is dense in $\{0\}\cup[2,\infty)$, even after restricting to groups of nilpotency class at most $2$. Moreover, we show that, despite being a commensurability invariant of nilpotent groups, conjugator length is not a quasi-isometry invariant: for every $m \geq 1$, the cocompact lattices in $\mathsf H(\mathbb R)^m$ realize exactly the growth types $n^2,\ldots,n^{m+1}$. Finally, for every real cubic algebraic number $\theta$, we construct a 2-step nilpotent group such that $n^3 \preceq \operatorname{CL}(n) \preceq n^{3+\varepsilon}$ for every $\varepsilon > 0$, with $n^3 \prec \operatorname{CL}(n)$ if and only if at least one real conjugate of $\theta$ has unbounded partial quotients. Consequently, determining the conjugator length function for arbitrary 2-step nilpotent groups is at least as hard as settling the bounded-partial-quotient problem for real cubic algebraic numbers with two nonreal conjugates.
\end{abstract}

\maketitle 

\setcounter{tocdepth}{1} 


\section{Introduction}

 \noindent  
Let \(G\) be a finitely generated group equipped with a word norm \(|\cdot|\). For conjugate elements \(u,v\in G\), let \(\CL(u,v)\) denote the minimum length of a \emph{conjugator} \(x \in G\) from $u$ to $v$; that is, an element \(x\in G\) such that \(x^{-1}ux=v\). The \emph{conjugator length function} of \(G\) is \[ \CL(n) := \max\bigl\{ \CL(u,v) \mid u,v\in G \text{ are conjugate and } |u|,|v|\leq n \bigr\}. \]
The conjugator length function is a natural quantitative measure associated with the conjugacy problem and has a geometric interpretation in terms of the geometry of annuli and the width of free homotopies; see \cite[\S2]{BRS26} for further motivation.
In the general finitely generated setting, independent recent work of Vandeputte and Goffer--Mihaila--Osin shows that a non-decreasing function occurs as a conjugator length function precisely if it is bounded or grows at least linearly \cite{Vandeputte2026,GMO26}. 

No analogous characterization is known for finitely presented groups. In particular, in their recent survey, Bridson--Riley--Sale pose the following question in the metabelian and polycyclic settings.

\begin{question}[{\cite[Open Problem 4.12]{BRS26}}] \label{Q:1}
Which functions arise as conjugator length functions of finitely presented metabelian groups or finitely generated polycyclic groups?
\end{question}

\noindent Earlier work established polynomial upper bounds for conjugator length in nilpotent groups \cite{JOR10,MMNV22}, while Sale obtained linear upper bounds in several metabelian and polycyclic families \cite{Sal16}. More recently,
Bridson--Riley realized \(n^d\), for every integer \(d\geq2\), in
both the model filiform groups and a family of 2-step nilpotent
groups \cite{BR26Filiform,BR26TwoStep}.

Finitely generated 2-step nilpotent groups are both finitely presented metabelian and polycyclic. 
Our first theorem produces the first nonintegral exponents in this intersection.

\begin{letterthm}\label{thm:denseCLspec}
	Let $\alpha \in [2,\infty) \cap \mathbb Q$. There exists a 2-step nilpotent group $G_\alpha$ with $\CL(n) \simeq n^\alpha$. 
\end{letterthm}

Together with a universal lower bound
\(n^c\preceq\CL(n)\) for nilpotent groups of virtual nilpotency class \(c\geq2\) (see \S\ref{subsec:univ-lower}), we will deduce the following result on the \emph{nilpotent conjugator length spectrum} --- the set of $\alpha \geq 0$ for which there is a nilpotent group with $\CL(n) \simeq n^\alpha$.

\begin{lettercor}
	The nilpotent conjugator length spectrum is dense in \(\{0\}\cup[2,\infty)\). This remains true after restricting to groups of nilpotency class \(\leq 2\). 
\end{lettercor}

In particular, no finitely generated nilpotent group has
\(\CL(n)\simeq n^\alpha\) for \(1<\alpha<2\). No finitely presented
group with \(\CL_G(n)\simeq n^\alpha\) for \(1<\alpha<2\) is currently
known. For
Dehn functions, the interval $(1,2)$ is known to be the only gap in the
exponent spectrum \cite{Ols91,Bow95,BB00}; Bridson--Riley--Sale ask whether, by contrast, conjugator-length exponents are dense
in \([1,2]\), see \cite[Open Problem~5.1]{BRS26}. 

For arbitrary finitely presented groups, nonintegral exponents are already known: Bridson--Riley constructed finitely presented groups with \(\CL(n)\simeq n^\alpha\) for a dense set of \(\alpha\in[2,\infty)\) \cite{BR25Snowflake}, after which Gillis--Wagner realized \(n^\alpha\) for every \(\alpha\geq2\) computable in double-exponential time \cite{GW26}. More generally, their construction shows that every Dehn function occurs, up to equivalence, as the conjugator length function of a finitely presented group.

\medskip 

In general, it is known that conjugator length can change on passing to a finite-index subgroup. Indeed, the infinite dihedral group $D_\infty = \mathbb Z \rtimes \mathbb Z/2\mathbb Z$ has $\CL(n) \simeq n$ while $\mathbb Z$ has $\CL = 0$. More strikingly, in \cite{GMO26} it is shown that any two non-decreasing functions of at least linear growth can be realized as the conjugator length functions of two commensurable finitely generated groups. 

The nilpotent setting is markedly more rigid. A central tool throughout the paper is that conjugator length is a commensurability invariant of nilpotent groups (Theorem~\ref{theorembody:commensurabilityinvariance}). By Mal'cev theory, the growth type of $\CL_G$ therefore depends only on the rational Mal'cev Lie algebra \(\mathfrak g_{\mathbb Q}\), which allows us to pass to Lie rings. It need not, however, depend only on the corresponding real Lie algebra \(\mathfrak g_{\mathbb R}\). This already occurs for direct powers of the real Heisenberg group $\mathsf H(\mathbb R)$.

\begin{letterthm}\label{thm:notQIinvariant}
	Let $m \geq 1$. The possible growth types of conjugator length functions of cocompact lattices in $\mathsf H(\mathbb R)^m$ are exactly \(
	\big\{ n^2, n^3, \ldots, n^{m+1} \big\}.
	\) 
	
	In particular, conjugator length is not a quasi-isometry invariant for nilpotent groups.
\end{letterthm}

It remains open whether a real nilpotent Lie algebra can admit
infinitely many rational forms with pairwise inequivalent conjugator length functions.

Our constructions are arithmetic in nature. Building on \cite{BR26TwoStep}, we use that conjugacy in class $2$ is governed by equations of the form $[x,u] = u-v$ in the Lie ring.  In coordinates, these become systems of integral linear equations in the unknown conjugator $x$. Consequently, conjugator length is governed by the worst-case size of a shortest integral solution among systems arising from bounded elements $u$ and $v$. We construct the rational Lie algebra for which these systems have the prescribed Diophantine behavior. 

For Theorems~\ref{thm:denseCLspec} and \ref{thm:notQIinvariant}, we implement this by using Galois-twisted rational forms of 2-step nilpotent real graph Lie algebras, whose description is proved by \cite{DW23}. The sizes of shortest solutions to the resulting conjugacy equations are controlled by the Archimedean sizes of units in the defining number fields, which we analyze using Dirichlet's unit theorem.

The same reduction reveals that conjugator length can depend on much subtler Diophantine information related to rational approximation and continued fractions. Recall the following classical conjecture.

\begin{conjecture}[Khinchin \cite{Khinchin1949ContinuedFractions}]\label{conjecture-intro:boundedPartialQuotients}
Every real algebraic number \(\theta\) of degree \(\geq3\) has
unbounded partial quotients in its continued fraction expansion.
\end{conjecture}

In fact, no real algebraic number of degree $\geq 3$ is known to have either bounded or unbounded partial quotients. We connect conjugator length in 2-step nilpotent groups to this conjecture.

\begin{letterthm}\label{thm:unboundedpartialquotients}  Let $\theta$ be a real algebraic number of degree $3$. There exists a 2-step nilpotent group $G_\theta$ such that
	\begin{enumerate}
		\item $n^3 \preceq \CL_{G_\theta}(n) \preceq n^{3 + \varepsilon}$ for every $\varepsilon > 0$, 
		\item $n^3 \prec \CL_{G_\theta}(n)$ if and only if a real conjugate of $\theta$ has unbounded partial quotients. 
	\end{enumerate}
\end{letterthm} 
When $\theta$ is an algebraic integer, the 2-step nilpotent group $G_\theta$ may be realized as a normal subgroup of Hirsch length 8 inside the Heisenberg group $H\big(\mathbb Z[\theta] \big)$, whose Hirsch length is $9$. 

 When $\theta$ has two nonreal conjugates, the second assertion in Theorem~\ref{thm:unboundedpartialquotients} detects whether $\theta$ itself has unbounded partial quotients. Consequently, any criterion deciding whether or not $\CL(n) \simeq n^3$ for arbitrary 2-step nilpotent groups would settle the bounded-partial-quotient problem for every real cubic algebraic number with two nonreal conjugates. Thus, already for 2-step nilpotent groups, a complete understanding of conjugator length encounters a classical open problem in Diophantine approximation. 
 
 \subsection{Question}
 Conditional on Conjecture~\ref{conjecture-intro:boundedPartialQuotients}, Theorem~\ref{thm:unboundedpartialquotients} provides the first nilpotent groups with $\CL(n)$ not equivalent to $n^\alpha$ for any $\alpha$. It remains open whether such behavior occurs unconditionally among nilpotent groups. 
 
 \begin{question}
 	Does there exist a nilpotent group with  $\CL(n) \not\simeq n^\alpha$ for any $\alpha$?
 \end{question}
 
 This contrasts with the state of knowledge for Dehn functions of nilpotent groups. Wenger constructed 2-step nilpotent groups whose Dehn functions satisfy $n^2 \prec \delta(n) \preceq n^2 \log n$ in \cite{Wen11}; some of these Dehn functions were recently shown to grow exactly like $n^2 \log n$ in \cite{Vandermeersch2026}. By contrast, it remains open whether the Dehn function of a nilpotent group can grow like $n^\alpha$ for a noninteger $\alpha$. 

\subsection{Notation}\label{subsec:notation} Throughout, $\mathbb N = \{1,2,\ldots\}$. For functions $f,g \colon \mathbb N \to  [0,\infty)$, we write \(f \preceq g\) if there exists an integer \(C\geq 1\) such that
\[
f(n)\le C\,g(Cn+C)+C
\qquad\text{for all } n\in \mathbb N.
\]
We write $f \simeq g$ if $f \preceq g \preceq f$, and write $f \prec g$ if $f \preceq g$ and $f \not\simeq g$.  

In proofs, we use $X \ll Y$ to denote the bound $X \leq C Y$, where $C \geq 1$ is a constant depending only on fixed ambient groups/Lie algebras and chosen norms. If additional dependence is allowed, it is denoted by a subscript. 
 We write $X \asymp Y$ if $X \ll Y \ll X$.

For a number field $K$ of degree $d=[K:\mathbb Q]$, let
$\sigma_1,\ldots,\sigma_d\colon K\hookrightarrow\mathbb C$ denote its
distinct field embeddings. We will often use the Archimedean
	supremum norm
\[
\|x\|_K:=\max_{1\leq i\leq d}|\sigma_i(x)|.
\]
With respect to any fixed $\mathbb Q$-basis of $K$, this norm is
bilipschitz equivalent to the corresponding coefficient supremum norm.
When $K$ is clear from context, we simply write $\|x\|$. 

\addtocontents{toc}{\protect\setcounter{tocdepth}{2}}
\section{Preliminary tools: commensurability invariance and Lie rings/algebras}\label{sec:prelims}

\noindent We call two finitely generated nilpotent groups $G_1$ and $G_2$ \emph{commensurable} if some finite-index subgroup of $G_1$ is isomorphic to some finite-index subgroup of $G_2$. In \S\ref{subsec:comm-invar}, we show that $\CL$ is a commensurability invariant of nilpotent groups (Theorem~\ref{theorembody:commensurabilityinvariance}).  
Therefore, up to $\simeq$, conjugator length depends only on the \emph{rational Mal'cev Lie algebra}~$\mathfrak g_{\mathbb Q}$, see \S\ref{sec:reductiontoLierings}.  

Throughout this paper, our main tool for computing conjugator length functions will be a reformulation of $\CL$ via a full Lie ring $L \subset \mathfrak g_{\mathbb Q}$, which we discuss in \S\ref{sec:reductiontoLierings}.  
In \S\ref{subsec:univ-lower}, we give an application of this: an efficient proof that every rational nilpotent Lie algebra of class \(c\geq2\) admits a uniform \(n^c\) lower bound on its conjugator length function (Proposition~\ref{propbody:universalLowerBound}). Finally, in \S\ref{subsec:rationalforms-graphalgebras}, we recall the classification \cite{DW23} of the rational forms of $c$-step nilpotent real Lie algebras associated to graphs. We will use these graph Lie algebras in the constructions for Theorems~\ref{thm:denseCLspec} and~\ref{thm:notQIinvariant}. 

\subsection{Commensurability invariance} \label{subsec:comm-invar}

\begin{theorem}\label{theorembody:commensurabilityinvariance}
	Let $G_1$ and $G_2$ be commensurable nilpotent groups. Then $\CL_{G_1} \simeq \CL_{G_2}$.
\end{theorem}

\noindent The proof has two parts. 
\begin{itemize}
\item First, we remove torsion. Proposition~\ref{prop:torsionfreequotientinvariance} shows that $\CL_G \simeq \CL_{G/\tau(G)}$ where $\tau(G)$ is the finite normal subgroup of torsion elements.
\item Then, we prove finite-index invariance. Proposition~\ref{prop:finiteindexinvariance} says that if $\Gamma$ is finitely generated torsion-free nilpotent and $H\leq \Gamma$ has finite index, then $\CL_\Gamma \simeq \CL_H$.
\end{itemize}
Together, Propositions~\ref{prop:torsionfreequotientinvariance} and~\ref{prop:finiteindexinvariance} imply Theorem~\ref{theorembody:commensurabilityinvariance}, since commensurable finitely generated nilpotent groups have commensurable torsion-free quotients. 

\medskip

\noindent \textit{Removing torsion.} \label{sec:torsion}
The subset of torsion elements $\tau(G)$ in a finitely generated nilpotent group $G$ is a finite normal subgroup. We call $\pi\colon G \to \Gamma := G / \tau(G)$ the torsion-free quotient. 

\begin{proposition}\label{prop:torsionfreequotientinvariance}
Let $G$ be a finitely generated nilpotent group. Then $\CL_{G} \simeq \CL_{G/\tau(G)}$. 
\end{proposition}

Fix word metrics $|\cdot |_G, |\cdot|_\Gamma$ for $G$ and $\Gamma = G / \tau(G)$. Since $\tau(G)$ is finite, there is a constant $L \geq 1$ such that 
\[
\frac 1 L \, |\pi(g)|_\Gamma \leq |g|_G \leq L \, |\pi(g)|_\Gamma + L \qquad \text{for all }g \in G. 
\]
\begin{lemma}
	\(\CL_\Gamma \preceq \CL_G\).
\end{lemma}
\begin{proof}
Let \(\bar a = \pi(a), \bar b = \pi(b) \in \Gamma \) be conjugate with $|\bar a |_\Gamma, |\bar b|_\Gamma \leq n$. Then there exists $t \in \tau(G)$ such that $a$ and $bt$ are conjugate in $G$. Since $\tau(G)$ is finite, the preceding estimates give $|a|_G, |bt|_G \leq Ln +L$.

Therefore, there exists a conjugator $z \in G$ from $a$ to $bt$ with 
$|z|_G \leq \CL_G(Ln + L).$ 
Since \(\pi(z)\) conjugates \(\bar a\) to \(\bar b\) in \(\Gamma\), and $|\pi(z)|_\Gamma \leq L |z|_G $, this proves $\CL_\Gamma(n) \preceq \CL_G(n).$
\end{proof}

For the converse inequality and for the difficult direction in finite-index invariance and the reduction to Lie rings in \S\ref{sec:reductiontoLierings}, we shall use the following lemma. 
Its proof passes to the real Mal'cev completion $\mathsf G_{\mathbb R}$ of $\Gamma$, bounds such representatives in terms of the covering radius of $C_\Gamma(g)$ in $C_{\mathsf G_{\mathbb R}}(g)$, and then bounds that radius by $\CL_\Gamma(|g|_\Gamma)$.

\begin{lemma}\label{lemma:centralizer-cosets-uniform}
	Let $M\geq 1$. There is a constant $K\geq 1$ such that for every $g \in \Gamma$, every coset of every subgroup $H\subset C_\Gamma(g)$ with $[C_\Gamma(g):H]\leq M$ contains a representative $c$ satisfying
	\[
	|c|_\Gamma \leq K\CL_\Gamma(K|g|_\Gamma+K)+K.
	\]
\end{lemma}
\begin{proof}
	Let $\mathsf G_{\mathbb R}$ be the real Mal'cev completion of $\Gamma$ (see \S\ref{subsubsec:malcev-completion} for its definition). Fix a left-invariant Riemannian metric $d$ on $\mathsf G_{\mathbb R}$. Since $\Gamma$ is a cocompact lattice in $\mathsf G_{\mathbb R}$, they are quasi-isometric and hence there exists $A\geq 1$ such that
	\[
	\frac1A \, d(1,g)-A\leq |g|_\Gamma\leq A \, d(1,g)+A
	\qquad\text{for all } g\in\Gamma,
	\]
	and every point of $\mathsf G_{\mathbb R}$ lies within distance $A$ of $\Gamma$.
	
	Fix $g\in\Gamma$. By Mal'cev theory, $C_{\mathsf G_{\mathbb R}}(g)$ is a connected rational Lie subgroup of $\mathsf G_{\mathbb R}$, and $C_\Gamma(g)$ is a cocompact lattice in $C_{\mathsf G_{\mathbb R}}(g)$ (see \cite[Lect.\ 1, Prop.\ 3.5]{Gel14}). Let $\Delta\geq 1$ be the least integer such that every point of $C_{\mathsf G_{\mathbb R}}(g)$ lies within ambient $d$-distance $\Delta$ of $C_\Gamma(g)$.
	
	We first bound bounded-index cosets in $C_\Gamma(g)$. Let $H\leq  C_\Gamma(g)$ have index $\leq M$, and let $c_0H$ be a left coset; the proof is identical for right cosets. Since $C_\Gamma(g)$ is generated by elements in a $d$-ball of radius $\leq 3 \Delta$ (see e.g.\ \cite[Theorem I.8.10(1)]{BH99}) and the coset $c_0H$ has a representative $c$ which is a product of at most $M$ such elements, we have that $d(1,c)\ll M\Delta$. Passing to the word metric on $\Gamma$ gives $|c|_\Gamma\leq K_1\Delta$, where $K_1 \geq 1$ depends only on $\Gamma$ and $M$.
	
	It remains to bound $\Delta$ in terms of $\CL_\Gamma(|g|_\Gamma)$. Choose $y\in C_{\mathsf G_{\mathbb R}}(g)$ with $d(y,C_\Gamma(g))\geq \Delta-1$, and choose $x\in\Gamma$ with $d(x,y)\leq A$. Put $b=x^{-1}gx$. Since $x$ differs from the centralizing element $y$ by a bounded element of $\mathsf G_{\mathbb R}$, we have $|b|_\Gamma\ll_A |g|_\Gamma+1$. 
	Let $z \in \Gamma$ be a shortest conjugator from $g$ to $b$. Then $zx^{-1}\in C_\Gamma(g)$, so $z=\gamma x$ for some $\gamma \in C_\Gamma(g)$. Therefore
	\[
	d(1,z)=d(\gamma^{-1},x)\geq d(x,C_\Gamma(g))\geq d(y,C_\Gamma(g))-A\geq  \Delta-A-1 . 
	\]
	The quasi-isometry $\Gamma \hookrightarrow \mathsf G_{\mathbb R}$ then yields $\Delta\ll_A  |z|_\Gamma+1$. Since $|b|_\Gamma\ll_A |g|_\Gamma+1$, it follows that 
	\[
	\Delta\leq K_2\CL_\Gamma(K_2|g|_\Gamma+K_2)+K_2
	\]
	for some $K_2 \geq 1$ depending only on $\Gamma$ and the fixed metrics. 	 
	Combining with $|c|_\Gamma\leq K_1\Delta$ proves the lemma. 
\end{proof}

\begin{lemma}\label{lemma:correction}
	There is a constant $K \geq 1$ such that if $gt$ is conjugate to $g$ for some $g \in G$ and $t \in \tau(G)$, then there exists a conjugator $z \in G$ from $gt$ to $g$ with 
	\[
	|z|_G\leq K\CL_\Gamma(K|\pi(g)|_\Gamma+K)+K.
	\]
\end{lemma}
\begin{proof}
	Write $\bar g=\pi(g)$. Then $H:=\pi(C_G(g))\leq  C_\Gamma(\bar g)$ satisfies 
	$[C_\Gamma(\bar g):H]\leq \# \tau(G)$. 
	
	Choose $x\in G$ with $x^{-1}(gt)x=g$. The projections of all conjugators
	from $gt$ to $g$ form the coset $\pi(x)H$ in $C_\Gamma(\bar g)$. Applying
	Lemma~\ref{lemma:centralizer-cosets-uniform} with $M=\# \tau(G)$, we deduce that this coset contains
	an element $\zeta$ such that 
	$
	|\zeta|_\Gamma\leq K\CL_\Gamma(K|\bar g|_\Gamma+K)+K.
	$ 
	Since $\tau(G)$ is finite, any conjugator $z$ from $gt$ to $g$ with $\pi(z) = \zeta$ satisfies $|z|_G\ll |\zeta|_\Gamma + 1$. Increasing $K\geq 1$ gives the
	claimed bound.
\end{proof}

\begin{lemma}\label{lemma:torsion-upperbound}
	\(\CL_G \preceq \CL_{\Gamma}\).
\end{lemma}

\begin{proof}
	Let $a,b\in G$ be conjugate with $|a|_G,|b|_G\leq n$. Then
	$|\pi(a)|_\Gamma,|\pi(b)|_\Gamma\ll n$. Choose
	$\gamma\in\Gamma$ with $\gamma^{-1}\pi(a)\gamma=\pi(b)$ and
	$|\gamma|_\Gamma\leq \CL_\Gamma(Cn+C)$. Choose $x\in G$ with $\pi(x) = \gamma$. Then $|x|_G\ll \CL_\Gamma(Cn+C)+1$ and $x^{-1}ax=bt$ for some
	$t\in \tau(G)$. Since $a$ is conjugate to $b$, the element $bt$ is conjugate to $b$.
	By Lemma~\ref{lemma:correction} above, there is $z\in G$ such that $z^{-1}(bt)z=b$ and
	$|z|_G\ll \CL_\Gamma(Cn+C)+1$. 
	Therefore $(xz)^{-1}a(xz)=b$, and
	$|xz|_G\ll \CL_\Gamma(Cn+C)+1$. This proves $\CL_G(n)\preceq \CL_\Gamma(n)$.
\end{proof}

\medskip 

\noindent \textit{$N$-saturated subsets and finite-index invariance.}  
Let $\Gamma$ be a finitely generated torsion-free nilpotent group equipped
with a word length $|\cdot|_\Gamma$. To prove finite-index invariance of $\CL$, we work slightly more generally by considering $N$\emph{-saturated subsets}, defined below. This will be useful when reducing to Lie rings in \S\ref{sec:reductiontoLierings}, where we prove a similar finite-index invariance statement for Lie subrings. 

 Let $N\triangleleft\Gamma$ be a
finite-index normal subgroup. We call a subset $S\subset\Gamma$ \emph{$N$-saturated} if it is an
inverse-closed union of $N$-cosets containing the identity coset $N$. Equivalently, 
\[
N\subset S=S^{-1},
\qquad
SN=S.
\]
Note that $S$ need not itself be a
subgroup. 

For $a,b\in S$, say that they are \emph{$S$-conjugate} if
$x^{-1}ax=b$ for some $x\in S$, and define
\[
\CL_S(a,b)
:=
\min\bigl\{|x|_\Gamma \: \big| \: x\in S,\ x^{-1}ax=b\bigr\}.
\]
We then define $\CL_S(n)$ by taking the maximum over all $S$-conjugate
$a,b\in S$ with $|a|_\Gamma,|b|_\Gamma\leq n$. Notice that
$S$-conjugacy need not be transitive.

\begin{lemma}\label{lemma:finite-model}
	Let $S\subset\Gamma$
	be an $N$-saturated subset. Then
	$
	\CL_S\simeq\CL_\Gamma. 
	$
\end{lemma}

\begin{proof}
	We first prove $\CL_\Gamma\preceq\CL_S$. Choose a finite set $R\subset\Gamma$
	of right coset representatives of $\Gamma / N$ and let $e$ be the exponent of $\Gamma / N$. 
	
	Let $a,b\in\Gamma$ be conjugate, with
	$|a|_\Gamma,|b|_\Gamma\leq n$. Choose a conjugator $x\in\Gamma$ and
	write $x=hr$, where $h\in N$ and $r\in R$. Put
	$d=rbr^{-1}$. Then
	$
	h^{-1}ah=d.
	$ 
	Since $a^e,d^e\in N\subset S$, the elements $a^e$ and $d^e$ are
	$S$-conjugate by $h$. Moreover,
	$|a^e|_\Gamma,|d^e|_\Gamma\ll n +1$. Hence there exists $y\in S$ such that
	$
	y^{-1}a^ey=d^e$ and $|y|_\Gamma \leq \CL_S(Cn+C).$  
	The torsion-free nilpotent group $\Gamma$ has unique roots, so
	$y^{-1}ay=d$. Consequently, $yr \in \Gamma$ conjugates $a$ to $b$, and therefore $\CL_\Gamma(n) \preceq \CL_S(n)$.

	Conversely, let $a,b\in S$ be $S$-conjugate, with
	$|a|_\Gamma,|b|_\Gamma\leq n$. Choose a shortest conjugator
	$x\in\Gamma$ from $a$ to $b$, so that
	$|x|_\Gamma\leq\CL_\Gamma(n)$, and choose any $S$-conjugator
	$x_0\in S$ from $a$ to $b$. Then
	$ 
	c_0:=x_0x^{-1}\in C_\Gamma(a).
	$ 
 Since
	$
	[C_\Gamma(a):C_\Gamma(a)\cap N]\leq[\Gamma:N],
	$ 
	Lemma~\ref{lemma:centralizer-cosets-uniform} gives 
	$c\in (C_\Gamma(a)\cap N)c_0$ and $K \geq 1$ such that
	\[
	|c|_\Gamma
	\leq K\CL_\Gamma(K|a|_\Gamma+K)+K.
	\]
	Write $c=hc_0$ with $h\in C_\Gamma(a)\cap N$. Then
	$
	cx=hx_0\in NS=S.
	$ 
	Moreover, $c$ centralizes $a$, so $cx$ is still a conjugator from
	$a$ to $b$. It follows that
	\[
	\CL_S(a,b)
	\leq |cx|_\Gamma
	\leq K\CL_\Gamma(Kn+K)+K +\CL_\Gamma(n).
	\]
	Thus $\CL_S(n)\preceq\CL_\Gamma(n)$.
\end{proof}

\begin{proposition}\label{prop:finiteindexinvariance}
	Let $\Gamma$ be a finitely generated torsion-free nilpotent group and
	let $H\leq\Gamma$ have finite index. Then
	$
	\CL_H\simeq\CL_\Gamma.
	$ 
\end{proposition}

\begin{proof}
	Take the normal core $
	N
	:=\bigcap_{\gamma\in\Gamma}\gamma H\gamma^{-1}
	$  of $H$ in $\Gamma$.
	Then $N\triangleleft\Gamma$ has finite index, and $H$ is an
	$N$-saturated subset. Lemma~\ref{lemma:finite-model} therefore shows that
	the conjugator length function of $H$, measured using
	$|\cdot|_\Gamma$, is equivalent to $\CL_\Gamma$. Since the restriction
	of $|\cdot|_\Gamma$ to $H$ is bilipschitz equivalent to any word length
	on $H$, the result follows.
\end{proof}

\subsection{Conjugator length in nilpotent Lie rings and Lie algebras}\label{sec:reductiontoLierings}

\noindent Theorem~\ref{theorembody:commensurabilityinvariance} shows that, for nilpotent groups, conjugator length depends only on the associated \emph{rational Mal'cev Lie algebra} $\mathfrak g_{\mathbb Q}$. Indeed, it is a classical result by Mal'cev that two finitely generated torsion-free nilpotent groups are commensurable if and only if the rational Mal'cev Lie algebras are $\mathbb Q$-isomorphic (\cite{Mal49,Mal51}, \cite[\S6A--C]{Seg83}, \cite[\S2]{Dek18}). 

Therefore, it is well-defined to speak of the conjugator length function of a given rational nilpotent Lie algebra $\mathfrak g_{\mathbb Q}$ up to $\simeq$. We now pass to Lie rings inside this Lie algebra, which will simplify the computations in the remainder of this paper. 

\subsubsection{The Mal'cev completion and its Lie algebra.}\label{subsubsec:malcev-completion} Let $\Gamma$ be a finitely generated torsion-free nilpotent group. We write $G_{\mathbb Q}$ for its \emph{rational Mal'cev completion}; this is the unique torsion-free divisible nilpotent group containing $\Gamma$ and such that every element in $G_{\mathbb Q}$ has a power lying in $\Gamma$. Let $\mathfrak g_{\mathbb Q}$ be the associated rational Mal'cev Lie algebra. We refer to \cite[\S2]{Dek18} for an elementary construction of $G_{\mathbb Q}$ and $\mathfrak g_{\mathbb Q}$. 

Let $\mathsf G_{\mathbb R}$ be the simply connected Lie group integrating $\mathfrak g_{\mathbb R}:=\mathfrak g_{\mathbb Q} \otimes_{\mathbb Q} \mathbb R $. Then we may view $\Gamma \subset G_{\mathbb Q} \subset \mathsf G_{\mathbb R}$. Another classical result by Mal'cev states that $\mathsf G_{\mathbb R}$, the \emph{real} Mal'cev completion, is the unique simply connected nilpotent Lie group containing $\Gamma$ as a cocompact lattice  (\cite[\S2]{Rag72}, \cite[\S2]{Dek18}). 

 The exponential map identifies $\mathfrak g_{\mathbb Q}$ (resp.\ $\mathfrak g_{\mathbb R}$) with $G_{\mathbb Q}$ (resp.\ $\mathsf G_{\mathbb R}$) via the \emph{Baker--Campbell--Hausdorff formula} (BCH), denoted by
 \[
X \ast Y := \log{\big(\exp(X)\cdot \exp(Y)\big)} = X + Y + \frac12 [X,Y] + \sum_{ k = 3}^c q_k(X,Y). 
\]
Here, the $q_k(X,Y)$ are rational linear combinations of right-normed $k$-fold Lie brackets in $X$ and $Y$. These $q_k(X,Y)$ are given by a universal formula and do not depend on \(\mathfrak g_{\mathbb Q}\); see \cite[\S6]{Seg83}. 

Under this identification, conjugation\footnote{The sign convention is obviously immaterial for computing conjugator length; in Lie rings it is more elegant to use the convention $e^{\ad X}(Y)$ instead of $e^{\ad (-X)}(Y)$.}  $xyx^{-1}$ corresponds to $
X*Y*(-X)=e^{\ad X}(Y),
$ 
where $\ad X(Y)=[X,Y]$ and \[e^{\ad X}(Y)=Y+[X,Y]+\frac1{2!}[X,[X,Y]]+\ldots+ \frac{1}{(c-1)!} [\underbrace{X,[\cdots[X, [X}_{\text{$c-1$ copies of $X$}},Y]] \cdots]].\] 

\subsubsection{Lie rings, Guivarc'h lengths, and conjugator length.} 
 A \emph{full Lie ring} in $\mathfrak g_{\mathbb Q}$ is a full $\mathbb Z$-lattice $L\subset \mathfrak g_{\mathbb Q}$ such that $[L,L]\subset L$. We say that $L$ is \emph{BCH-closed} if $x*y\in L$ for all $x,y\in L$. In that case, $\Gamma = \exp(L)$ is a finitely generated torsion-free nilpotent group, specifically an \emph{LR-group}, with rational Mal'cev Lie algebra $\mathfrak g_{\mathbb Q}$. 

Every Lie ring $L\subset \mathfrak g_{\mathbb Q}$ contains a finite-index subring of the form $DL$ which is BCH-closed by taking $D \geq 1$ sufficiently large (in order to clear the denominators in the BCH formula). Thus any full Lie ring in $\mathfrak g_{\mathbb Q}$ gives, up to commensurability, a group with rational Mal'cev Lie algebra $\mathfrak g_{\mathbb Q}$.

\medskip 

Following \cite{Gui73}, we put a Guivarc'h length $\|\cdot \|_{\mathrm{Gv}}$ on $L$ and $\mathfrak g_{\mathbb Q}$ as follows. Let
$\gamma_1(\mathfrak g_{\mathbb Q})=\mathfrak g_{\mathbb Q}$,
$\gamma_{i+1}(\mathfrak g_{\mathbb Q})=[\mathfrak g_{\mathbb Q},\gamma_i(\mathfrak g_{\mathbb Q})]$
be the lower central series. Choose subspaces $V_i$ such that
$\gamma_i(\mathfrak g_{\mathbb Q})=V_i\oplus \gamma_{i+1}(\mathfrak g_{\mathbb Q})$, and choose vector space norms $\|\cdot \|_i$ on $V_i\otimes_{\mathbb Q}\mathbb R$ (or, more generally, quasi-norms bilipschitz equivalent to such vector space norms) and restrict them to $V_i$. Define 
\[
\|X\|_{\mathrm {Gv}}=\max_i \|v_i\|_i^{1/i}, \qquad X = v_1+\ldots +v_c, \quad v_i \in V_i.
\]
By \cite{Gui73}, the resulting Guivarc'h length $\|\cdot \|_{\mathrm{Gv}}$ on $L$ is independent of these choices up to bilipschitz equivalence. 

\medskip 

We now define conjugator length directly on a full Lie ring.  Let $L \subset \mathfrak g_{
\mathbb Q}$ be a full Lie ring. 
We call $u,v \in L$ \emph{conjugate} in $L$ if there exists $x \in L$ such that $e^{\ad x}(u) = v$.

	For conjugate $u,v\in L$, define the \emph{conjugator length}
	$\CL_L(u,v)\geq 0$ as the Guivarc'h length of a shortest conjugator from
	$u$ to $v$. 
	The \emph{conjugator length function} of $L$ is
	\[
	\CL_L(n):=\max\bigl\{\,\CL_L(u,v) \, \big| \, u,v\in L \text{ conjugate},\
	\|u\|_{\mathrm {Gv}},\|v\|_{\mathrm {Gv}}\leq n\,\bigr\}.
	\]
	The $\simeq$-class of $\CL_L$ does not depend on the chosen Guivarc'h norm.

\medskip 

One caveat of this generalization is that $e^{\ad x}(u)$ need not lie in $L$
for arbitrary $u,x\in L$ if $L$ is not BCH-closed. Thus, if $L$ is not
BCH-closed, conjugacy in $L$ need not be transitive, although it is reflexive
and symmetric.

\begin{lemma}\label{lemma:BCH-liering-group-CL}
	Assume $L$ is BCH-closed and let $\Gamma := \exp(L)$. Then $\CL_{\Gamma} \simeq \CL_L$.
\end{lemma}

\begin{proof}
	It is proved in \cite{Gui73} that under exponential coordinates, every word length on $\Gamma$ is bilipschitz equivalent to every Guivarc'h length on $L$, and conjugacy in $\Gamma$ is exactly the relation $e^{\ad x}(u) = v$ in $L$. Hence $\CL_\Gamma(\exp u, \exp v) \asymp \CL_L(u,v)$ and  $\CL_{\Gamma} \simeq \CL_L$.
\end{proof}

The next proposition shows that BCH-closedness is only a convenience:
every full Lie ring in $\mathfrak g_{\mathbb Q}$ gives the same
conjugator length function.

\begin{proposition}\label{prop:full-liering-CL}
	Let $L\subset\mathfrak g_{\mathbb Q}$ be a full Lie ring. Then
	$
	\CL_L\simeq\CL_{\mathfrak g_{\mathbb Q}}.
	$ 
\end{proposition}

\begin{proof}
 Choose an integer $D\geq1$ divisible by all denominators occurring in the BCH formula and in the formula for $e^{\ad x}$ up to the class $c$. Put
	\[
	L_i:=L\cap\gamma_i(\mathfrak g_{\mathbb Q}),
	\qquad
	A:=\sum_{i=1}^c D^{-(2i-1)}L_i.
	\]
	Then $A$ is a full Lie ring containing $L$. Moreover, it is closed
	termwise under all Lie polynomials occurring in the BCH and conjugation
	formulas. In particular, $A$ is BCH-closed.  
	
	Choose $q\geq1$ such that $B:=qA\subset L$. Then $B$ is BCH-closed and
	\[
	e^{\ad a}(B)=B
	\qquad\text{for every }a\in A,
	\]
	that is, $B$ is both a Lie ideal and a normal subgroup of $A$ under BCH. 
	Furthermore, if $b\in B$ and $x\in L$, every non-linear term in
	$b*x-x$ and $x*b-x$ contains at least one occurrence of $b$ and
	therefore lies in $B$. Hence
	\[
	B*L=L*B=L.
	\]
	
	Now put $
	\Gamma:=\exp(A)$, $N:=\exp(B)$ and $S:=\exp(L)$.  
	Then $N\triangleleft\Gamma$ has finite index and, in the terminology of Lemma~\ref{lemma:finite-model},  $S \subset \Gamma$ is an $N$-saturated subset. Lemma~\ref{lemma:finite-model} gives $\CL_\Gamma \simeq\CL_S$. As in the proof of Lemma~\ref{lemma:BCH-liering-group-CL}, we conclude that 
	$
	\CL_L \simeq \CL_\Gamma$. 
	Finally, $\Gamma$ has rational Mal'cev Lie algebra
	$\mathfrak g_{\mathbb Q}$, so
	$\CL_\Gamma\simeq\CL_{\mathfrak g_{\mathbb Q}}$.
\end{proof}

\subsection{Universal $n^c$ lower bound} \label{subsec:univ-lower}
The virtual nilpotency class $c$ of a finitely generated nilpotent group $G$ is the minimal nilpotency class among all finite-index subgroups. Equivalently, it is the nilpotency class of the torsion-free quotient \(\Gamma := G / \tau(G)\), of its rational Mal'cev Lie algebra, and of its real Mal'cev Lie algebra. 
It is a famous theorem of Gersten--Holt--Riley that $c$-step nilpotent groups have a universal $n^{c+1}$ upper bound on their Dehn functions \cite{GHR03}. In this section, we prove a universal $n^c$ lower bound on conjugator length in (non-virtually Abelian) nilpotent groups.

\begin{proposition}\label{propbody:universalLowerBound} If $G$ is nilpotent of virtual nilpotency class $c \geq 2$, then $n^c \preceq \CL_G(n)$. 
\end{proposition}
The proof uses ideas similar to those used in the model filiform lower-bound argument in \cite{BR26Filiform}. 
\begin{proof} 
It suffices to prove \(n^c\preceq\CL_L(n)\) for a full Lie ring \(L\) in a rational nilpotent Lie algebra~\(\mathfrak g_{\mathbb Q}\) of class \(c\geq2\).
	
	Equip $L \subset \mathfrak g_{\mathbb Q}$ with a Guivarc'h length $\|\cdot \|_{\mathrm{Gv}}$. There exist $a \in L$ and $u \in \gamma_{c-1}(L) \subset \gamma_{c-1}(\mathfrak g_{\mathbb Q})$ such that $z := [a,u] \neq 0$ is central. Put $v_n := u + n^c z$ for $n\geq 1$. Then 
	\[
	e^{\ad n^c a}(u) = u + [n^c a, u] = u+ n^c z = v_n,
	\]
	hence $u$ and $v_n$ are conjugate in $L$. Moreover, $\| u\|_{\mathrm{Gv}} \ll 1$ and $ \|v_n\|_{\mathrm{Gv}} \ll (n^c)^{1/c}=n$. 
	
	It remains to show that any conjugator from $u$ to $v_n$ has length $\gg n^c$. Let $\|\cdot\|_1$, $\|\cdot\|_c$ be vector space norms on $\mathfrak g_{\mathbb Q} / \gamma_2(\mathfrak g_{\mathbb Q})$ and $\gamma_c(\mathfrak g_{\mathbb Q})$ and consider the linear map
	\[
	\beta\colon \mathfrak g_{\mathbb Q} / \gamma_2(\mathfrak g_{\mathbb Q}) \longrightarrow \gamma_c(\mathfrak g_{\mathbb Q}), \qquad \beta\big( \bar x \big) = [x,u]. 
	\]
	If $x\in L$ conjugates $u$ to $v_n$, then $e^{\ad x}(u)=  u + [x,u] =v_n$ implies $[x,u]=n^cz$. Since $\beta$ is linear, it is Lipschitz continuous; thus
	$
 n^c \asymp \|n^cz\|_c = \|\beta(\bar x)\|_c \ll \|\bar x \|_1 .
	$ 
	Since $\|\bar x \|_1 \ll \|x\|_{\mathrm{Gv}}$, this concludes the proof.
\end{proof}

\subsection{Rational forms of real graph Lie algebras} \label{subsec:rationalforms-graphalgebras} The final tool that we will use in this paper is the Deré--Witdouck rational form classification of $c$-step nilpotent Lie algebras associated to graphs \cite{DW23}. Here we recall their classification. 

\medskip 

Let $G$ be a finite simple undirected graph, let $c\geq2$, and let $\mathfrak n_{G,c}(\mathbb R)$ be the \emph{$c$-step nilpotent real Lie algebra associated to $G$}. More precisely, $\mathfrak n_{G,c}(\mathbb R)$ is the free $c$-step nilpotent real Lie algebra generated by the vertices of $G$, quotiented by the ideal generated by the brackets $[\alpha,\beta]$ for which $\{\alpha,\beta\}$ is not an edge of $G$. Unless $G$ is edgeless, this Lie algebra has nilpotency class exactly $c$.

Two vertices $\alpha, \beta$ of $G$ are called \emph{coherent} if the transposition $(\alpha \,\, \beta)$ swapping them defines a graph automorphism of $G$. This is an equivalence relation on the vertex set; its equivalence classes are called the \emph{coherent components}.

Then one defines the quotient graph $\overline{G}$, whose vertices are the coherent components. In contrast to $G$, this quotient is weighted by the number of vertices in each coherent component and may have loops. Two coherent components, possibly equal, are joined by an edge in $\overline{G}$ whenever at least one edge occurs between them in $G$. This is well-defined by \cite[Lemma 2.4]{DW23}: if one edge occurs between two coherent components, then all possible edges between them occur. An automorphism of $\overline G$ preserves its weights and loops.

The following theorem classifies the rational forms of the real graph Lie algebras $\mathfrak n_{G,c}(\mathbb R)$ for $c \geq 2$.  
For a finite totally real Galois extension $L/\mathbb Q$ and a faithful action $\rho\colon \operatorname{Gal}(L/\mathbb Q) \hookrightarrow \operatorname{Aut}(\overline G)$, Deré--Witdouck define a rational form $\mathfrak n_{\rho,c}({\mathbb Q})$ of $\mathfrak n_{G,c}(\mathbb R)$ as the twisted fixed-point algebra 
\[\mathfrak n_{\rho,c}({\mathbb Q}) := \big\{ v\in \mathfrak n_{G,c}(L) \:\big|\: i\big(\rho(\sigma)\big)(\sigma v) = v \text{ for all } \sigma \in \operatorname{Gal}(L/\mathbb Q)\big\},\]
where 
\begin{itemize}
	\item $i\colon \operatorname{Aut}(\overline{G}) \hookrightarrow \operatorname{Aut}(\mathfrak n_{G,c}(L))$ is the injective group morphism induced by a chosen fixed splitting $r\colon \operatorname{Aut}(\overline{G}) \hookrightarrow \operatorname{Aut}(G)$ of the natural projection $\operatorname{Aut}(G)\to \operatorname{Aut}(\overline{G})$, 
	\item $\operatorname{Gal}(L/\mathbb Q)$ acts naturally and coefficientwise on $\mathfrak n_{G,c}(L)$. 
\end{itemize}
\cite[Thm.\ A]{DW23} states that all rational forms of $\mathfrak n_{G,c}(\mathbb R)$ are precisely of this form.
\begin{theorem}\label{thm:rationalforms-graphalgebras}
	Every rational form of $\mathfrak n_{G,c}(\mathbb R)$ is $\mathbb Q$-isomorphic to a Lie algebra $\mathfrak n_{\rho,c}({\mathbb Q})$ for some finite totally real Galois extension $L/\mathbb Q$ and a faithful action $\rho\colon \operatorname{Gal}(L/\mathbb Q)\hookrightarrow \operatorname{Aut}(\overline{G}).$ Conversely, every such pair $(L,\rho)$ defines a rational form of $\mathfrak n_{G,c}(\mathbb R)$. 
	
	Moreover, given two such pairs $(L,\rho)$, $(L',\rho')$, the induced rational forms are $\mathbb Q$-isomorphic if and only if $L=L'$ as subfields of $\mathbb R$ and there exists $\varphi \in \operatorname{Aut}(\overline{G})$ such that \[
	\rho'(\sigma)  	= \varphi \, \rho(\sigma)\,  \varphi^{-1} \qquad \text{for all } \sigma \in \operatorname{Gal}(L/\mathbb Q).
	\]
\end{theorem}

\addtocontents{toc}{\protect\setcounter{tocdepth}{1}}

\section{No quasi-isometry invariance (Theorem~\ref{thm:notQIinvariant})} \label{sec:noQI}

\noindent Let $m \geq 1$ and let $\mathsf H(\mathbb R)^m= \mathsf H(\mathbb R) \times \cdots \times \mathsf H(\mathbb R)$ be the $m^\text{th}$ direct power of the real Heisenberg group. 

\setcounter{letterthmbody}{2}

\begin{letterthmbody}\label{thmbody:notQIinvariant}
	Let $m \geq 1$. The possible growth types of conjugator length functions of cocompact lattices in $\mathsf H(\mathbb R)^m$  are exactly \(
	\big\{ n^2, n^3, \ldots, n^{m+1} \big\}.
	\) 
	
	In particular, conjugator length is not a quasi-isometry invariant for nilpotent groups.
\end{letterthmbody}

To obtain Theorem~\ref{thmbody:notQIinvariant}, by the Lie and Mal'cev correspondences, it suffices to classify the conjugator length functions of all rational forms of the Lie algebra $\mathfrak h(\mathbb R)^{m}$  of $\mathsf H(\mathbb R)^m$, where $\mathfrak h(\mathbb R)$ is the $3$-dimensional real Heisenberg Lie algebra
\[
\mathfrak h(\mathbb R) = \mathbb R X \oplus \mathbb R Y \oplus \mathbb R Z, \qquad [aX,bY] = abZ.
\]

When $K$ is a number field of degree $d = [K : \mathbb Q]$, we will always view \[
\mathfrak h(K) = K X \oplus K Y \oplus K Z, \qquad [aX,bY] = ab Z
\] 
as a $3d$-dimensional $\mathbb Q$-Lie algebra by restricting scalars. 

\medskip 

In order to prove Theorem~\ref{thmbody:notQIinvariant}, it suffices to first classify all rational forms of $\mathfrak h(\mathbb R)^{ m}$ and then compute their conjugator length functions.

\begin{proposition}[Rational forms of $\mathfrak h(\mathbb R)^{ m}$]\label{prop:rationalFormsOfh(R)m}
	The rational forms of $\mathfrak h(\mathbb R)^{ m}$ are exactly the $\mathbb Q$-Lie algebras of the form \[
	\mathfrak h(K_1) \oplus \ldots \oplus \mathfrak h(K_s),
	\]
	for totally real number fields $K_1, \ldots, K_s$ with $[K_1 : \mathbb Q] + \ldots + [K_s:\mathbb Q] = m$. 
	
	Two such rational forms $\bigoplus_{i=1}^s \mathfrak h(K_i)$ and $\bigoplus_{i = 1}^t \mathfrak h(L_i)$ are $\mathbb Q$-isomorphic if and only if $s = t$ and, after reordering, $K_i \cong L_i$. 
\end{proposition}

\begin{proposition}\label{prop:conjugatorlengthfunctionofbigoplush(K_i)}
	The rational form $\bigoplus_{i=1}^s \mathfrak h(K_i)$ of $\mathfrak h(\mathbb R)^{ m}$ has $\CL(n) \simeq n^{d + 1}$, with $d = \max_i [K_i:\mathbb Q]$.  
\end{proposition}

Once Propositions~\ref{prop:rationalFormsOfh(R)m} and \ref{prop:conjugatorlengthfunctionofbigoplush(K_i)} are proved, Theorem~\ref{thmbody:notQIinvariant} follows: to realize every exponent $d+1$ for $d\in \{1,\ldots,m\}$, take $\mathfrak h(K) \oplus \mathfrak h (\mathbb Q)^{m-d}$, where $K$ is a totally real number field of degree $d$.

\begin{remark}
	This family of Lie algebras yields many natural 2-step nilpotent groups satisfying the main theorem of \cite{BR26TwoStep}. For each \(m\geq1\), take a totally real number field \(K\) of degree \(m\) and choose a finitely generated torsion-free nilpotent group \(\Gamma\) whose associated rational Mal'cev Lie algebra is \(\mathfrak h(K)\). Then \(\CL_\Gamma(n)\simeq n^{m+1}\).
\end{remark}

\begin{remark}[The case $m=2$] By \cite[Proposition 3.2]{Lau08} or \cite[Example 5.2]{DW23}, the rational forms of $\mathfrak h(\mathbb R)\oplus \mathfrak h(\mathbb R)$ are exactly the $\mathbb Q$-Lie algebras \[
	\mathfrak h(\mathbb Q) \oplus \mathfrak h(\mathbb Q),  \qquad \mathfrak h\big( \mathbb Q(\sqrt d) \big)  \quad \text{for square-free }d > 1.
	\]
	By Proposition \ref{prop:conjugatorlengthfunctionofbigoplush(K_i)}, $\CL_{\mathfrak h(\mathbb Q) \oplus \mathfrak h(\mathbb Q)}(n) \simeq n^2$ while $\CL_{\mathfrak h( \mathbb Q(\sqrt d))}(n) \simeq n^3$. 
\end{remark}

\subsection{Rational forms of $\mathfrak h(\mathbb R)^{ m}$ (Proof of Proposition~\ref{prop:rationalFormsOfh(R)m})} \label{subsec:rationalforms}

The Lie algebra $\mathfrak h(\mathbb R)^{ m}$ has basis $x_1, y_1, z_1, \ldots, x_m, y_m, z_m$ and nonzero brackets \[[x_1,y_1] = z_1, \quad \ldots \quad, [x_m,y_m] =z_m.\] 
With the conventions of \S\ref{subsec:rationalforms-graphalgebras}, $\mathfrak h(\mathbb R)^{ m}$ is the 2-step nilpotent Lie algebra associated to the simple undirected graph $G_m$ with vertices $x_1,y_1,\ldots,x_m,y_m$ and edges $\{x_1,y_1\}, \ldots, \{x_m,y_m\}$. The following figure illustrates the graph $G_m$ and its quotient graph $\overline{G_m}$:  

\begin{center}
	\begin{tikzpicture}[x=0.6cm,y=0.6cm,>=Latex]
		
		\tikzset{
			dot/.style={circle, fill=black, inner sep=2.2pt},
			qvtx/.style={circle, draw=black, fill=white, line width=0.8pt, minimum size=6mm, inner sep=0pt},
			every node/.style={font=\normalsize}
		}
		
		\node at (5.05,2.1) {$G_m$};
		\node at (17.2,2.1) {$\overline{G_m}$};
		
		\node[dot] (a1) at (1.6,1.0) {};
		\node[dot] (a2) at (1.6,-1.0) {};
		\draw[line width=0.8pt] (a1) -- (a2);
		
		\node[dot] (b1) at (4.0,1.0) {};
		\node[dot] (b2) at (4.0,-1.0) {};
		\draw[line width=0.8pt] (b1) -- (b2);
		
		\node at (6.1,0) {$\cdots$};
		
		\node[dot] (c1) at (8.5,1.0) {};
		\node[dot] (c2) at (8.5,-1.0) {};
		\draw[line width=0.8pt] (c1) -- (c2);
		
		\draw[line width=0.9pt,->] (10.3,0) -- (12.3,0);
		
		\draw[line width=0.9pt] (13.9,-0.2)
		to[in=50,out=130,loop,min distance=16mm,looseness=12]
		(13.9,-0.2);
		\node[qvtx] (q1) at (13.9,-0.4) {$2$};

		
		\draw[line width=0.9pt] (16.4,-0.2)
		to[in=50,out=130,loop,min distance=16mm,looseness=12]
		(16.4,-0.2);
		\node[qvtx] (q2) at (16.4,-0.4) {$2$};
		
		\node at (18.4,0) {$\cdots$};
		
		
		\draw[line width=0.9pt] (20.5,-0.2)
		to[in=50,out=130,loop,min distance=16mm,looseness=12]
		(20.5,-0.2);
		\node[qvtx] (q3) at (20.5,-0.4) {$2$};
		
	\end{tikzpicture}
\end{center}

Here we will apply Deré--Witdouck's Theorem~\ref{thm:rationalforms-graphalgebras} from \S\ref{subsec:rationalforms-graphalgebras} to deduce Proposition~\ref{prop:rationalFormsOfh(R)m}.

\begin{proof}[Proof of Proposition~\ref{prop:rationalFormsOfh(R)m}]
	The coherent components of $G_m$ are the pairs
	$\{x_i,y_i\}$, for $1\leq i\leq m$. Hence the quotient graph
	$\overline{G_m}$ consists of $m$ disconnected looped vertices of weight $2$, and therefore $
	\operatorname{Aut}(\overline{G_m})\cong\operatorname{Sym}(m).
	$
	We use the natural splitting $\operatorname{Aut}(\overline{G_m}) \hookrightarrow \operatorname{Aut}(G_m)$ under which a permutation of
	$\{1,\ldots,m\}$ simultaneously permutes the triples
	$(x_j,y_j,z_j)$. Thus it permutes the $m$ summands of
	$\mathfrak n_{G_m,2}(L)=\mathfrak h(L)^m$.
	
	By Theorem~\ref{thm:rationalforms-graphalgebras}, every rational form
	is isomorphic to $\mathfrak n_{\rho,2}(\mathbb Q)$ for a finite totally real Galois extension $L/\mathbb Q$ and a faithful action 
	$
	\rho\colon\operatorname{Gal}(L/\mathbb Q)
	\hookrightarrow\operatorname{Sym}(m)
	$ 
	on $\Omega:=\{1,\ldots,m\}$. We write $\sigma(j):= \rho(\sigma)(j)$ for this action. Write
	$\Omega=\Omega_1\sqcup\cdots\sqcup\Omega_s$ for its orbit
	decomposition. Choose $j_i\in\Omega_i$ and put
	\[
	H_i:=\operatorname{Stab}_{\operatorname{Gal}(L/\mathbb Q)}(j_i),
	\qquad
	K_i:=L^{H_i}.
	\]
	
	Denote the $j^{\text{th}}$ summand of $\mathfrak h(L)^m$ by $\mathfrak h(L)_j$ and write an element of $\mathfrak h(L)^m$ as
	$v=(v_j)_{j\in\Omega}$. Under the twisted Galois action,
	$\sigma$ sends the coordinate $v_j \in \mathfrak h(L)_j$ to $\sigma(v_j) \in \mathfrak h(L)_{\sigma(j)}$. Consequently,
	\[
	v\in\mathfrak n_{\rho,2}(\mathbb Q)
	\quad\Longleftrightarrow\quad
	v_{\sigma (j)}=\sigma(v_j)
	\quad\text{for all $\sigma\in\operatorname{Gal}(L/\mathbb Q)$ and $j\in\Omega$}.
	\]
	
	Consider $j_i\in\Omega_i$. Since $\Omega_i$ is the orbit of $j_i$, every
	$j\in\Omega_i$ has the form $j=\sigma(j_i)$ for some
	$\sigma\in\operatorname{Gal}(L/\mathbb Q)$. For
	$v\in\mathfrak n_{\rho,2}(\mathbb Q)$, the fixed-point condition then
	gives 
	$ 
	v_j=v_{\sigma(j_i)}=\sigma(v_{j_i}).
	$ 
	Hence all coordinates $v_j$, with $j\in\Omega_i$, are determined by
	the single coordinate $v_{j_i}$. Moreover, if $h\in H_i$, then $h(j_i)=j_i$, and therefore
	$
	v_{j_i}=v_{h(j_i)}=h(v_{j_i}).
	$ 
	Thus
	$ 
	v_{j_i}\in\mathfrak h(L)^{H_i}
	=\mathfrak h(K_i).
	$
	
	Conversely, let $u_i\in\mathfrak h(K_i)$. For $j\in\Omega_i$, choose
	$\sigma$ such that $j=\sigma(j_i)$ and define
	$v_j:=\sigma(u_i)$. If also $j=\tau(j_i)$, then
	$\tau^{-1}\sigma\in H_i$, and since $u_i$ is fixed by $H_i$, the definition of $v_j$ is independent of the choice of
	$\sigma$.
	
	It follows that the map
	\[
	\bigoplus_{i=1}^s\mathfrak h(K_i)
	\longrightarrow\mathfrak n_{\rho,2}(\mathbb Q),
	\qquad
	(u_i)_{i=1}^s\longmapsto(v_j)_{j\in\Omega},
	\qquad
	v_{\sigma(j_i)}:=\sigma(u_i),
	\]
	is a $\mathbb Q$-Lie algebra isomorphism; here the preservation of the
	bracket follows because the bracket is taken coordinatewise. Finally,
	each $K_i$ is totally real, and the orbit-stabilizer theorem gives
	$[K_i:\mathbb Q]=[\operatorname{Gal}(L/\mathbb Q):H_i]=\#\Omega_i$. Hence
	$\sum_i[K_i:\mathbb Q]=m$.
	
	\medskip 
	
	Conversely, let $K_1,\ldots,K_s \subset \mathbb R$ be totally real number fields whose
	degrees sum to $m$. Let $L$ be the smallest number field containing their Galois closures $\widetilde{K}_i$; then $L/\mathbb Q$ is Galois and totally real. Put $H_i=\operatorname{Gal}(L/K_i).$ 
	The natural action of $\operatorname{Gal}(L/\mathbb Q)$ on
	$\Omega := \coprod_i \operatorname{Gal}(L/\mathbb Q)/H_i$, where $\# \Omega = m$, is faithful, since its
	kernel is the intersection of the normal cores 
	\[
	\bigcap_{i=1}^s\operatorname{core}_{\operatorname{Gal}(L/\mathbb Q)}(H_i)
	= \bigcap_{i=1}^s 	\operatorname{Gal}\bigl(L/\widetilde K_i\bigr) = 
	\operatorname{Gal}\bigl(L/\widetilde K_1\cdots\widetilde K_s\bigr)
	=1,
	\]
	where $\widetilde K_i$ is the Galois closure of $K_i$. The associated
	rational form is therefore
	$\bigoplus_i\mathfrak h(K_i)$.
	
	\medskip 
	
	Finally, Theorem~\ref{thm:rationalforms-graphalgebras} says that two such forms are isomorphic precisely when their corresponding Galois fields coincide, say $L$, and the resulting $\operatorname{Gal}(L/\mathbb Q)$-actions are conjugate in $\operatorname{Sym}(m)$. Such a conjugacy identifies orbits and point stabilizers. Since $L^{\gamma H\gamma^{-1}}=\gamma(L^H)$, the corresponding fixed fields are $\mathbb Q$-isomorphic. Thus $s=t$ and, after reordering, $K_i\cong_{\mathbb Q}L_i$.

	Conversely, if $K_i \cong  L_i$ after reordering, then clearly \[\bigoplus_{i=1}^s \mathfrak h(K_i) \cong_{\mathbb Q} \bigoplus_{i=1}^s \mathfrak h(L_i). \qedhere \] 
\end{proof}

\subsection{Conjugator length function of $\bigoplus_{i=1}^s \mathfrak h(K_i)$ (Proof of Proposition~\ref{prop:conjugatorlengthfunctionofbigoplush(K_i)})}
It is easy to see that the conjugator length function of a direct product of groups is equivalent to the maximum of the conjugator length functions of the factors, see \cite[\S4.11]{BRS26}. 

Consequently, to prove Proposition~\ref{prop:conjugatorlengthfunctionofbigoplush(K_i)}, it suffices to show that for a fixed totally real number field $K$ of degree $d = [K:\mathbb Q]$, the $\mathbb Q$-Lie algebra $\mathfrak h(K)$ has $\CL(n) \simeq n^{d + 1}$. The case $d = 1$ is already known: the discrete Heisenberg group $H(\mathbb Z)$ has rational Mal'cev Lie algebra $\mathfrak h(\mathbb Q)$ and $\CL_{H(\mathbb Z)}(n) \simeq n^2$; see \cite[Thm.~4.9]{BRS26}. So we may assume that $d \geq 2$. 

\medskip

\subsubsection{Upper bound.} We apply one of the main results of Bridson--Riley \cite[Thm.~3]{BR26TwoStep}, restated below in our notation.
\begin{theorem} \label{thm:BR26TwoStepUpperBound}
	Any 2-step nilpotent rational Lie algebra $\mathfrak g_{\mathbb Q}$ has $\CL(n) \preceq n^{\dim Z(\mathfrak g_{\mathbb Q}) + 1}$. 
\end{theorem}
\begin{proof}
	Let $\Gamma$ be any finitely generated torsion-free nilpotent group with associated rational Mal'cev Lie algebra $\mathfrak g_{\mathbb Q}$. Since the torsion-free rank of $Z(\Gamma) \cong \mathbb Z^d$ equals $d = \dim Z(\mathfrak g_{\mathbb Q})$, the upper bound $\CL(n) \preceq n^{d + 1}$ follows from applying \cite[Thm.~3]{BR26TwoStep} to $\Gamma$.
\end{proof}

\begin{lemma}
	The rational Lie algebra $\mathfrak h(K)$ has $\CL(n) \preceq n^{d+1}$.
\end{lemma}
\begin{proof}
	This follows from Theorem~\ref{thm:BR26TwoStepUpperBound} since $Z\big(\mathfrak h(K)\big) = KZ$ is $d$-dimensional over $\mathbb Q$.
\end{proof}

\subsubsection{Lower bound.}\label{subsec:lowerboundhK} To prove a matching lower bound, we work in the full Lie ring \[
L:= \mathfrak h(\mathcal O) = \mathcal O X \oplus \mathcal O Y \oplus \mathcal O Z, \qquad [aX,bY] = ab Z,
\]
where $\mathcal O$ denotes the ring of integers of $K$. Choose the Minkowski embedding $K \hookrightarrow \mathbb R^d$ given by the $d$ real embeddings $\sigma_1 = \mathrm{id}, \sigma_2, \ldots, \sigma_d \colon K \hookrightarrow \mathbb R$. Define the norm
\[
\| a \| := \max_i | \sigma_i(a) |, \qquad a \in K
\]
and the Guivarc'h norm
\[
\|aX + bY + cZ \|_{\mathrm{Gv}} := \max \big\{ \|a\|, \|b\|, \|c\|^{1/2} \big\}, \qquad aX + bY + cZ \in L.
\]
To prove the lower bound, we need a quantitative consequence of Dirichlet's unit theorem, for which we first recall the following context (see \cite[Ch.\ 1 \S7]{Neu99}). Consider the logarithmic embedding 
$
\mathrm{Log}\colon \mathcal O^\times \to \mathbb R^d$ defined by $\operatorname{Log}(\varepsilon) = \big(\log |\sigma_1(\varepsilon)|, \ldots, \log | \sigma_{d}(\varepsilon)| \big).
$ 
Since $\varepsilon$ is a unit, we have $\sum_{i =1}^{d} \log |\sigma_i(\varepsilon)| = \log | N_{K / \mathbb Q}(\varepsilon)| = 0.$ Thus $\operatorname{Log}(\mathcal O^\times)$ lies in the $(d-1)$-dimensional vector subspace \[H := \Big\{(t_1,\ldots,t_{d}) \in \mathbb R^d \mid \sum_i t_i = 0\Big\}.\] 

\begin{theorem}[Dirichlet's unit theorem] \label{thm:DirichletUnit}
	The Abelian group $\operatorname{Log}(\mathcal O^\times)$ is a full lattice in $H$.
\end{theorem}

\begin{lemma}\label{lemma:dirichlet}
	Assume $d \geq 2$. There exists a sequence of units $\varepsilon_n \in \mathcal O^\times$ such that \[
	\| \varepsilon_n \| \asymp n, \qquad \| \varepsilon_n^{-1} \| \asymp n^{d-1}. 
	\]
\end{lemma}
\begin{proof} By Dirichlet's unit theorem~\ref{thm:DirichletUnit}, here with $
	\operatorname{Log}(\varepsilon) = \big(\log |\sigma_1(\varepsilon)|, \ldots, \log |
	\sigma_d(\varepsilon)| \big)
	$ and $H = \{(t_1,\ldots,t_d)\in \mathbb R^d \, | \, \sum_i t_i = 0\}$, there exists a constant $R > 0$ such that every element of $H$ lies within $R$-distance of $\operatorname{Log}(\mathcal O^\times)$. For $t > 0$, consider $v_t = (-(d-1)t, t, \ldots, t) \in H$. Choose units $\varepsilon_t$ such that $\| \operatorname{Log}(\varepsilon_t) - v_t\| \leq R$. Put $n = e^t$. Then 
	\[
	|\sigma_1(\varepsilon_t)| \asymp n^{-(d-1)}, \qquad |\sigma_i(\varepsilon_t)| \asymp n, \quad 2 \leq i \leq d.
	\]
	Therefore $\|\varepsilon_t \|= \max_i |\sigma_i(\varepsilon_t)| \asymp n$ and $\| \varepsilon_t^{-1} \| =   \big(\min_i |\sigma_i(\varepsilon_t)|\big)^{-1} \asymp n^{d-1}.$  
	Reindexing by $n$ proves the lemma.
\end{proof}

\begin{lemma}
	The full Lie ring $L \subset \mathfrak h(K)$ has $n^{d+1} \preceq \CL(n)$.
\end{lemma}
\begin{proof}
	Fix $n \geq 1$. Choose $\varepsilon_n \in \mathcal O^\times$ as in Lemma~\ref{lemma:dirichlet}. Let $\delta = n^2 \in \mathbb Z \subset \mathcal O$. We will show that the elements \[
	u_n = \varepsilon_n X, \qquad v_n = \varepsilon_n X + \delta Z
	\]
	have Guivarc'h length $\ll n$, that they are conjugate in $L$, and that every conjugator has length $\gg n^{d+1}$.
	
	\medskip 
	
	The first statement is straightforward: $\|\varepsilon_n\| \asymp n$ and $\|\delta \| = n^2$, so the Guivarc'h length of both elements is $\asymp n$. 
	
	Next, any conjugator $w = p X + qY + rZ \in L$ from $u_n$ to $v_n$ must satisfy 
	\[
	u_n + [w,u_n] = v_n \iff \varepsilon_n X - q \varepsilon_n Z = \varepsilon_n X + \delta Z \iff -q \varepsilon_n = \delta.
	\]
	Since $\varepsilon_n$ is a unit, this equation has a solution, e.g.\ $w = - \varepsilon_n^{-1} \delta \, Y $. Thus $u_n$ and $v_n$ are conjugate. Moreover, any such conjugator $w$ has Guivarc'h length $\gg n^{d+1}$ since its $Y$-coordinate must satisfy \[
	\|q\| = \|\delta \varepsilon_n^{-1} \| = |\delta| \cdot \| \varepsilon_n^{-1} \| \asymp n^2 \cdot n^{d-1} = n^{d+1}.
	\]
	This proves the lower bound.
\end{proof}

\section{Dense nilpotent conjugator length spectrum (Theorem~\ref{thm:denseCLspec})}

\noindent This section is dedicated to the proof of the following theorem.

\setcounter{letterthmbody}{0}

\begin{letterthmbody}\label{thmbody:denseCLspec}
	Let $\alpha \in [2,\infty) \cap \mathbb Q$. There exists a 2-step nilpotent rational Lie algebra $\mathfrak g_{\alpha}$ with $\CL(n) \simeq n^\alpha$. 
\end{letterthmbody}
Then Theorem~\ref{theorembody:commensurabilityinvariance}, Proposition~\ref{propbody:universalLowerBound} and Theorem~\ref{thmbody:denseCLspec} together immediately imply:

\begin{lettercorbody}
	The nilpotent conjugator length spectrum is dense in \(\{0\}\cup[2,\infty)\). This remains true after restricting to groups of nilpotency class \(\leq 2\). 
\end{lettercorbody}
\begin{remark}[Virtually Abelian case]
Theorem~\ref{theorembody:commensurabilityinvariance} and  Proposition~\ref{propbody:universalLowerBound} imply that a finitely generated nilpotent group $G$ has conjugator length function $\CL_G(n) \simeq 1$ if and only if $G$ is virtually Abelian. 
\end{remark}

\subsection{Proof of Theorem~\ref{thmbody:denseCLspec}}

Since the discrete Heisenberg group $H(\mathbb Z)$ has quadratic conjugator length function (see \cite[Thm.\ 4.9]{BRS26}), to prove Theorem~\ref{thmbody:denseCLspec}, it suffices to focus on the case $\alpha > 2$. 

Let $\alpha \in (2,\infty) \cap \mathbb Q$ and write $\alpha = 1 + {r}/{2k}$
for some fixed $r,k \geq 1$ with $r \geq 2k+3$. This can always be done: if $\alpha = 2 + a/b$ for some $a,b \geq 1$, take $r = 4(a+b)$ and $k = 2b$.

  We construct a 2-step nilpotent rational Lie algebra $\mathfrak g_{r,k}$ with $\CL(n) \simeq n^{1 + {r}/{2k}}$.  
Let $K$ be a totally real Galois extension of $\mathbb Q$ with \[\operatorname{Gal}(K/\mathbb Q) = \langle \gamma \rangle \cong \mathbb Z / r \mathbb Z.\] 
For every $r \geq 1$, such an extension exists by \cite[Lemma~5.4]{DW23}.

Define the $r(k+1)$-dimensional rational Lie algebra \[
\mathfrak{g}_{r,k} := K X \oplus \bigoplus_{s = 1}^k K Z_s,
\]
with all $K Z_s$ central and only non-trivial bracket relations \[
[aX, bX] = \sum_{s=1}^k \big( a \gamma^s(b) - b \gamma^s(a) \big) Z_s, \qquad a,b\in K. 
\]
The goal of this section is to prove that $\CL_{\mathfrak g_{r,k}}(n) \simeq n^{1 + r / (2k)}$.  For completeness, we first prove that $\mathfrak g_{r,k}$ is a rational form of a real graph Lie algebra.

\subsubsection{$\mathfrak g_{r,k}$ is a rational form of a graph algebra.} The relevant real Lie algebra is the 2-step nilpotent graph Lie algebra associated to $C_r^k$, the $k^\text{th}$ power of the $r$-cycle graph $C_r$. That is, take the $r$-cycle graph $C_r$ and add edges between those vertices whose distance along $C_r$ is $\leq k$. Equivalently, its set of vertices is $\mathbb Z / r \mathbb Z$ and vertex $i$ is connected to the vertices $i \pm 1, i \pm 2, \ldots, i \pm k$. 
Two examples of such graphs are shown below. 
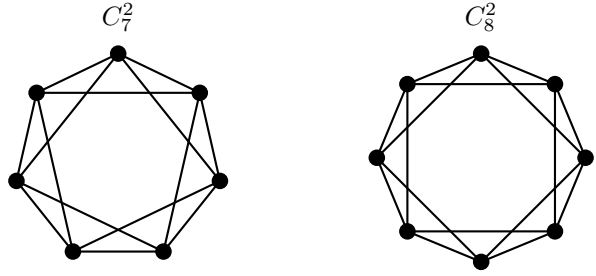
\begin{center}
	\begin{tikzpicture}[
		x=0.6cm,
		y=0.6cm,
		line cap=round,
		line join=round
		]
		
		\tikzset{
			dot/.style={circle, fill=black, inner sep=2.2pt},
			every node/.style={font=\normalsize}
		}
		
		\begin{scope}[shift={(-4,0)}]
			
			\node at (0,3.1) {$C_7^2$};
			
			\foreach \i in {0,...,6}{
				\coordinate (v\i) at ({90-360*\i/7}:2.3);
			}
			
			\foreach \i in {0,...,6}{
				\pgfmathtruncatemacro{\j}{mod(\i+1,7)}
				\pgfmathtruncatemacro{\k}{mod(\i+2,7)}
				\draw[line width=0.8pt] (v\i) -- (v\j);
				\draw[line width=0.8pt] (v\i) -- (v\k);
			}
			
			\foreach \i in {0,...,6}{
				\node[dot] at (v\i) {};
			}
			
		\end{scope}
		
		\begin{scope}[shift={(4,0)}]
			
			\node at (0,3.1) {$C_8^2$};
			
			\foreach \i in {0,...,7}{
				\coordinate (w\i) at ({90-360*\i/8}:2.3);
			}
			
			\foreach \i in {0,...,7}{
				\pgfmathtruncatemacro{\j}{mod(\i+1,8)}
				\pgfmathtruncatemacro{\k}{mod(\i+2,8)}
				\draw[line width=0.8pt] (w\i) -- (w\j);
				\draw[line width=0.8pt] (w\i) -- (w\k);
			}
			
			\foreach \i in {0,...,7}{
				\node[dot] at (w\i) {};
			}
			
		\end{scope}
		
	\end{tikzpicture}
\end{center}

For $F \in \{K,\mathbb R\}$, we denote the 2-step nilpotent $F$-Lie algebra associated to $C_r^k$ by $\mathfrak n_{r,k}(F)$. It has basis $x_i, z_{i,s}$ for $i \in \mathbb Z/r\mathbb Z, 1 \leq s \leq k$ and nonzero brackets \[
[x_i,x_{i+s}] = z_{i,s}, \qquad [x_i,x_{i-s}] = -z_{i-s,s}. 
\]
For $r \geq 2k + 3$, the coherent components of $C_r^k$ are the singletons. Hence $\overline{C_r^k} = C_r^k$ and $\operatorname{Aut}\big( \overline{C_r^k} \big) = \operatorname{Aut}\big(C_r^k\big)$ in the notation of Deré--Witdouck (\S\ref{subsec:rationalforms-graphalgebras}).  
We now show that the Lie algebra $\mathfrak g_{r,k}$ is a Galois-twisted rational form of $\mathfrak n_{r,k}(\mathbb R)$.

\begin{lemma}
	The $\mathbb Q$-Lie algebra $\mathfrak g_{r,k}$ is the rational form of $\mathfrak n_{r,k}(\mathbb R)$ corresponding to the cyclic rotation action $ \operatorname{Gal}(K/\mathbb Q)\to \operatorname{Aut}\big( C_r^{k}\big)$ under the correspondence of Theorem~\ref{thm:rationalforms-graphalgebras}.
\end{lemma}
\begin{proof}
	The cyclic rotation action gives a semilinear Galois action on $\mathfrak n_{r,k}(K)$ defined by \[
	\gamma(\lambda x_i) = \gamma(\lambda) x_{i+1}, \qquad \gamma(\lambda z_{i,s}) = \gamma(\lambda) z_{{i+1},s}. 
	\] 
	This is compatible with the bracket. 
	
	We prove that the corresponding fixed-point $\mathbb Q$-Lie algebra \[
	\left( \mathfrak n_{r,k}(K) \right)^{\operatorname{Gal}(K / \mathbb Q)} = \big\{ u \in  \mathfrak n_{r,k}(K) \: \big| \: \gamma(u) = u \big\}
	\]
	is exactly $\mathfrak g_{r,k}$. This suffices: indeed, by the standard Galois descent theorem for Lie algebras, recalled, for example, in \cite[\S3.1]{DW23}, the fixed-point Lie algebra of a semilinear Galois action on a finite-dimensional $K$-Lie algebra is a $\mathbb Q$-form of the original algebra. It then follows that $
	\mathfrak g_{r,k} \otimes_{\mathbb Q} \mathbb R \cong \mathfrak n_{r,k}(K) \otimes_K \mathbb R \cong \mathfrak n_{r,k}(\mathbb R).
	$
	
	We now show $\mathfrak g_{r,k} \cong \left( \mathfrak n_{r,k}(K) \right)^{\operatorname{Gal}(K / \mathbb Q)}$. Define a \(\mathbb Q\)-linear map \[
\Theta\colon \mathfrak g_{r,k} \to  \left( \mathfrak n_{r,k}(K) \right)^{\operatorname{Gal}(K / \mathbb Q)}, \quad	\Theta(aX) = \sum_{i \in \mathbb Z / r \mathbb Z} \gamma^i(a) x_i, \quad \Theta(cZ_s) = \sum_{i \in \mathbb Z / r \mathbb Z} \gamma^i(c) z_{i,s}. 
	\]
	This map is well-defined: it is straightforward to verify that $\gamma \big( \Theta(aX) \big) =  \Theta(aX)$ and $\gamma \big( \Theta(c Z_s) \big) = \Theta(c Z_s).$ 
	We show that $\Theta$ is an isomorphism of $\mathbb Q$-Lie algebras. First, it is compatible with the brackets: since $Z_s \in Z(\mathfrak g_{r,k})$ and $z_{i,s} \in Z\big(\mathfrak n_{r,k}(K)\big)$, it suffices to check that 
	\begin{equation}
		\Theta\big([aX, bX]\big) = \big[\Theta(aX), \Theta(bX)\big], \qquad a,b \in K. \label{eq:compatibleWithTheBracket}
	\end{equation}
	The left-hand side is
	\[
	\Theta\big([aX, bX]\big) = \Theta \left( \sum_{s=1}^k \big( a \gamma^s(b) - b \gamma^s(a) \big) Z_s  \right) =  \sum_i \sum_{s=1}^k \big( \gamma^i(a) \gamma^{i+s}(b) - \gamma^i(b) \gamma^{i+s}(a) \big) z_{i,s}, 
	\]
	while the right-hand side is 
\begin{align*}
	\big[\Theta(aX),\Theta(bX)\big]
	&=
	\left[
	\sum_i\gamma^i(a)x_i,
	\sum_j\gamma^j(b)x_j
	\right]
	\\
	&=
	\sum_i\sum_{s=1}^k
	\Big(
	[\gamma^i(a)x_i,\gamma^{i+s}(b)x_{i+s}]
	+
	[\gamma^{i+s}(a)x_{i+s},\gamma^i(b)x_i]
	\Big).
\end{align*}
	Thus \eqref{eq:compatibleWithTheBracket} holds. 
	
	It remains to show $\Theta$ is bijective. Since $\left( \mathfrak n_{r,k}(K) \right)^{\operatorname{Gal}(K / \mathbb Q)}$ is a $\mathbb Q$-form of the $K$-Lie algebra $\mathfrak n_{r,k}(K)$, and $\dim_K \mathfrak n_{r,k}(K) = \dim_{\mathbb Q} \mathfrak g_{r,k} = r(k+1)$, it follows that $\mathfrak g_{r,k}$ and $\left( \mathfrak n_{r,k}(K) \right)^{\operatorname{Gal}(K / \mathbb Q)}$ have the same dimension. It therefore suffices to show $\Theta$ is injective. 
	
	Let $aX + \sum_{s = 1}^k c_s Z_s \in \operatorname{Ker}(\Theta)$. Then \[\Theta \left( aX + \sum_{s = 1}^k c_s Z_s \right) = \sum_{i } \gamma^i(a) x_i + \sum_{s = 1}^k \sum_{i } \gamma^i(c_s) z_{i,s} = 0.
	\]
	The coefficient of $x_0$ is $a$, so $a = 0$. Similarly, the coefficient of $z_{0,s}$ is $c_s$, so $c_s = 0$. Hence the kernel is zero. This proves the lemma.
\end{proof}

%

\subsubsection{Lie ring $L\subset \mathfrak g_{r,k}$ and norms.}
\label{subsec:Lieringandnorms}

To prove that $\CL_{\mathfrak g_{r,k}}(n) \simeq n^{1 + r/(2k)}$, we work in the class-2 full Lie ring \[
L := \mathcal O X \oplus \bigoplus_{s = 1}^k \mathcal O Z_s \subset \mathfrak g_{r,k},
\]
where $\mathcal O$ denotes the ring of integers of $K$. 

Notice that the $\sigma_i :=  \gamma^i$ for $i\in \mathbb Z/r\mathbb Z$ define the distinct real embeddings $K \hookrightarrow \mathbb R$ and define the Archimedean supremum norm \[
\| x \| = \max_i | \sigma_i(x) |, \qquad x \in K. 
\]
On $L$, use the Guivarc'h norm \[
\left\| aX + \sum_{s = 1}^k c_s Z_s \right\|_{\mathrm{Gv}} = \max \left\{\| a\| , \|c_1\|^{1/2}, \ldots, \|c_k\|^{1/2}\right\}.
\]

\subsubsection{Conjugacy equation and conjugator length function.}\label{subsec:Conjugacyequationandconjugatorlengthfunction} Let 
\[
u = bX + \sum_{s = 1}^k c_s Z_s, \quad v = bX + \sum_{s = 1}^k c_s Z_s + \delta \in L,
\] where $\delta = \sum_{s = 1}^k \delta_s Z_s.$ Then $u$ and $v$ are conjugate if and only if there exists $a \in \mathcal O$ such that \[
u + [aX, u] = v \iff \sum_{s=1}^k \big(a\gamma^s(b) - b \gamma^s(a) \big) Z_s = \delta \iff \: \forall 1 \leq s \leq k\colon \:\: a \gamma^s(b) - b \gamma^s(a) = \delta_s.
\]
Let $B_b\colon \mathcal O  \to  \mathcal O^k $ be the $\mathbb Z$-linear map defined by \[
B_b(a) = \big(a \gamma(b) - b \gamma(a), a \gamma^2(b) - b \gamma^2(a) , \ldots,a \gamma^k(b) - b \gamma^k(a) \big). 
\]
The conjugacy equation becomes $B_b(a) = \delta$ and conjugator length can be rephrased as follows:  
\begin{equation}
	\CL_{L}(n) \simeq \max_{\substack{\|b\|\leq n\\
			\delta\in B_b(\mathcal O)\\
			\|\delta\|\leq n^2}}
	\min\big\{\|a\|: B_b(a)=\delta \big\}, \label{eq:equivalentCL}
\end{equation} 
where $\mathcal O^k$ is equipped with the sup-norm $\|(\delta_1,\ldots,\delta_k)\| = \max_s \| \delta_s\|$. 

For $b\neq0$, the homogeneous conjugacy equation has a
rank-1 $\mathbb Z$-module of solutions:
\begin{lemma}\label{lemma:kernel}
	$\operatorname{Ker} B_b = \mathcal O \cap b \mathbb Q$ for $b \neq 0$.
\end{lemma}
\begin{proof}
	The $\supset$-inclusion is straightforward. For the $\subset$-inclusion, let $a \in \mathcal O$ with $B_b(a) = 0$. By the first equation, $a \gamma(b) - b \gamma(a) = 0$, so $\gamma(a/b) = a/b$ in $K$. 
	Since $\gamma$ generates $\operatorname{Gal}(K/\mathbb Q)$, its fixed field is $\mathbb Q$. Hence $a/b \in \mathbb Q$. This completes the proof.
\end{proof}

\subsubsection{Lower bound $n^{1 + r/(2k)} \preceq \CL_L(n) $.} \label{subsec:lowerbounddense} Fix $c = (r-2k)/(2k) > 0$. Then $2 + c = 1 + r/(2k)$. In this subsection, we focus on proving the target lower bound $n^{2 + c}$.

As in the lower-bound argument for $\mathfrak h(K)$ in \S\ref{subsec:lowerboundhK}, we use Dirichlet's unit theorem to construct a sequence of units $b_n$ with prescribed Archimedean sizes. From this we obtain a central displacement $\delta_n$ with norm $\ll n^2$, such that the conjugacy equation $B_{b_n}(a) = \delta_n$ has only solutions of norm $\succeq n^{2+c}$.

\begin{lemma}
	For $n \geq 1$, we construct elements $a_n \in \mathcal O, b_n \in \mathcal O^\times$ such that:
	\begin{enumerate}
		\item There exist  $m,M > 0$ independent of $n$ such that \begin{align*}
			&mn^{-c}  \leq 	|\sigma_i(b_n)| \leq M n^{-c} \quad \text{for $i \in \{\pm 1, \ldots, \pm k\}$}, \\
			&mn \leq |\sigma_i(b_n)| \leq M n \quad \: \: \text{for all other $i$}.
		\end{align*}
		\item 	$
		|\sigma_0(a_n)| \ll 1, \quad \big| \sigma_i(a_n) - \sigma_i(b_n)n^{1+c} \big| \ll 1 \quad$ for $i \neq 0$. 
	\end{enumerate}
\end{lemma}
\begin{proof}
 By Dirichlet's unit theorem~\ref{thm:DirichletUnit}, there exists \(R>0\) such that every element of \(H\) lies within \(R\)-distance of \(\operatorname{Log}(\mathcal O^\times)\). For \(n\geq1\), take the vector \(v_n\in\mathbb R^r\) defined by \[
 (v_n)_i = - c \log n \quad \text{for $i \in \{\pm1, \ldots, \pm k\}$} \quad \text{and} \quad (v_n)_i = \log n \quad \text{for $i \notin \{\pm1, \ldots, \pm k\}$}.
 \]
	Since \(2k(-c)+(r-2k)=0\), we have \(v_n\in H\). Choose \(b_n\in\mathcal O^\times\) such that
	\(\|\operatorname{Log}(b_n)-v_n\|\leq R\). Exponentiating the coordinatewise estimates gives 
	\[
	|\sigma_i(b_n)|\asymp n^{-c}\quad \text{for }i\in \{\pm1,\ldots,\pm k\} \quad \text{and} \quad 
	|\sigma_i(b_n)|\asymp n\quad \text{for } i\notin \{\pm1,\ldots,\pm k\}.
	\]
	
	For~(2), we use that the Minkowski image of \(\mathcal O\) is a lattice in \(\mathbb R^r\):  there exists \(R'>0\) such that every point of \(\mathbb R^r\) lies within \(R'\)-distance of this lattice \cite[Ch.~1 \S5]{Neu99}. Apply this to the vector \(w_n\in\mathbb R^r\) defined by \((w_n)_0=0\) and
	\((w_n)_i=\sigma_i(b_n)n^{1+c}\) for \(i\neq0\). We obtain \(a_n\in\mathcal O\) such that
	\(\|(\sigma_i(a_n))_i-w_n\|\leq R'\). Consequently,
	\[
	|\sigma_0(a_n)|\ll1 \quad \text{and} \quad 
	\big|\sigma_i(a_n)- \sigma_i(b_n) n^{1+c} \big|\ll1 \quad
	\text{for }i\neq0. \qedhere 
	\]
\end{proof}

Define $\delta_n = B_{b_n}(a_n)$. Since $\|b_n\| \ll n$, Eq.\ \eqref{eq:equivalentCL} shows that it suffices to prove $\|\delta_n\| \ll n^2$ and that every solution of $B_{b_n}(a) = \delta_n$ satisfies $n^{2+c} \ll \|a\|$ for sufficiently large $n$. 

\begin{lemma}\label{lemma:boundDelta}
	$\| \delta_n\| \ll n^2$.
\end{lemma}
\begin{proof}
	For $1 \leq s \leq k$, the $s$-component of $\delta_n = B_{b_n}(a_n)$ is $a_n \gamma^s(b_n) - b_n \gamma^s(a_n)$.  Thus, \[
	\big|\sigma_i\big(a_n \gamma^s(b_n) - b_n \gamma^s(a_n)\big) \big| = \big|\sigma_i(a_n) \sigma_{i+s}(b_n) - \sigma_{i}(b_n) \sigma_{i+s}(a_n)\big|. 
	\]
	It suffices to show this quantity has size $\ll n^2$ for all $i \in \mathbb Z/ r\mathbb Z, 1 \leq s \leq k$. 
	
	If neither $i$ nor $i+s$ equals $0$, then by construction of $a_n$, 
	\begin{align*}
		\big|\sigma_i(a_n) &\sigma_{i+s}(b_n) - \sigma_{i}(b_n) \sigma_{i+s}(a_n)\big| \\ &= \big| \big(\sigma_{i}(b_n)n^{1+c} + O(1)\big) \cdot \sigma_{i+s}(b_n) - \sigma_{i}(b_n) \cdot \big(\sigma_{i+s}(b_n)n^{1+c} + O(1)\big) \big| \\ &\ll |\sigma_{i+s}(b_n)| + |\sigma_{i}(b_n)| \ll n. 
	\end{align*}
	
	If $i = 0$, then  $
	\big|\sigma_0(a_n) \sigma_{s}(b_n) - \sigma_{0}(b_n) \sigma_{s}(a_n)\big|  \ll 1 \cdot n^{-c} +  n \cdot n^{-c} \cdot n^{c+1} \ll n^2.
	$ 
	The case $i + s = 0$ is analogous. This proves the lemma. 
\end{proof}

The following lemma completes the proof of the lower bound $n^{1 + r/(2k)} \preceq \CL_L(n) $. 

\begin{lemma}\label{lemma:boundA}
There exists $n_0 \geq 1$ such that for every $n \geq n_0$ and every $a \in \mathcal O$ with $B_{b_n}(a) = \delta_n$ it holds that $n^{2+c} \ll \|a\| $.
\end{lemma}
\begin{proof}
	Since $B_{b_n}(a - a_n) = \delta_n - \delta_n = 0$, Lemma \ref{lemma:kernel} implies that $a = a_n + q b_n$ for some $q \in \mathbb Q$. 	
	To prove $\|a\| = \max_i |\sigma_i(a)| \gg n^{2+c}$, we bound $|\sigma_i(a_n + q b_n)|$ for $i = 0$ and $i = j$, where $j$ is any non-neighbor of $0$, for instance $j = k+1$.

	By construction of $a_n$, 
	\[
	|\sigma_0(a_n)| \leq C, \qquad \sigma_j(a_n) = \sigma_j(b_n)n^{1+c} + R_{j,n} \quad \text{for $|R_{j,n}| \leq C$},	
	\]
	for some $C \geq 1$ independent of $n$. 
	
	First look at the $0$-vertex:
	\[
	\left|	\sigma_0(a_n + q b_n) \right| = \left| \sigma_0(b_n) \right| \left| q + \frac{\sigma_0(a_n)}{\sigma_0(b_n)} \right| \geq mn \left| q + \frac{\sigma_0(a_n)}{\sigma_0(b_n)} \right|.
	\]
	Similarly, at the $j$-vertex:
	\[
	\left|	\sigma_j(a_n + q b_n) \right| =\left| \sigma_j(b_n)(q+ n^{1+c}) + R_{j,n} \right|  \geq mn \left| q + n^{1+c} + \frac{R_{j,n}}{\sigma_j(b_n)} \right|.
	\]
	Now we compare the two quantities on the right-hand side and bound their difference:
\[
		\left|q + \frac{\sigma_0(a_n)}{\sigma_0(b_n)} - \left( q + n^{1+c} + \frac{R_{j,n}}{\sigma_j(b_n)} \right) \right| = \left| n^{1+c}  + \frac{R_{j,n}}{\sigma_j(b_n)} - \frac{\sigma_0(a_n)}{\sigma_0(b_n)} \right| \geq  n^{1+c} - 2 \frac{C}{mn} \geq \frac12 n^{1+c}\]
	for all $n \geq n_0$, where $n_0$ depends only on $C,m,c$.  
	Using the identity $\max\{|X|,|Y|\} \geq \frac12 |X - Y|$, we conclude that
	\[
	\|a\| \geq \max\{ |\sigma_0(a)|, |\sigma_j(a)| \} \geq mn \cdot \frac12 \left| q + \frac{\sigma_0(a_n)}{\sigma_0(b_n)} - \left( q + n^{1+c} + \frac{R_{j,n}}{\sigma_j(b_n)} \right)  \right|  \geq \frac m4 \,  n^{2+c}\]
	for all $n \geq n_0$. 
\end{proof}

\subsubsection{Upper bound $\CL_L(n) \preceq n^{1 + r/(2k)}$.} \label{subsec:upperbounddense}  

The following weighted path lemma on $C_r^k$ is the combinatorial heart of the proof of the upper bound. 

\begin{lemma}\label{lemma:weightedPathLemma} Let $x_0, \ldots,x_{r-1} \geq 0$. Then for every pair of vertices $p,q \in \mathbb Z /r\mathbb Z$, there is a path from $p$ to $q$ in $C_r^k$ such that every edge $\{i,j\}$ on the path satisfies \[ x_i + x_j \leq \max\{x_p,x_q\} + \frac {\sum_i x_i}{2k}. \] \end{lemma}
\begin{proof} 	First, we argue that the graph $C_r^k$ is $\geq \! 2k$-vertex-connected, where the vertex connectivity $\kappa(G)$ of a graph $G$ is the minimum number of vertices whose deletion disconnects $G$ or leaves only one vertex. The following statement (see \cite[Thm.~2]{Hobbs73}) relates the connectivity of the $r$-cycle graph $C_r$ to that of its $k^{\mathrm{th}}$ power: \[
	\kappa(C_r^k) \geq \min \big\{ \# C_r -1, k \cdot \kappa(C_r)\big\}. 
	\]
	Clearly, $\kappa(C_r) = 2$. Since $r \geq 2k+3$, we deduce $\kappa(C_r^k) \geq \min\{r-1,2k\} = 2k$, showing that $C_r^k$ is $\geq \! 2k$-vertex-connected. 
	
	\medskip 
	
	Assume $p \neq q$. Put $S := \sum_i x_i$. 
	Call an edge $\{i,j\}$ \emph{bad} if
	\[
	x_i+x_j>\max\{x_p,x_q\}+\frac{S}{2k}.
	\]
	Aiming for a contradiction, suppose that every path joining $p$ to $q$ in $C_r^k$ contains a bad edge.  	

By the global form of Menger's theorem \cite[Thm.~3.3.6(i)]{Diestel17}, since $C_r^k$ is $\geq \! 2k$-connected, the vertices
$p$ and $q$ are joined by at least $2k$ pairwise internally vertex-disjoint paths
$
P_1,\ldots,P_{2k}.
$ 
In other words, for $\ell\neq \ell'$, the only vertices that
$P_\ell$ and $P_{\ell'}$ have in common are $p$ and $q$.

Define $a_\ell := \sum_{i \in V(P_\ell)\setminus\{p,q\}} x_i$ if $P_\ell$ is not the one-edge path $p$--$q$, and define $a_\ell := \min\{x_p,x_q\}$ if $P_\ell$ is $p$--$q$.  
By internal vertex-disjointness, we see that $\sum_{\ell=1}^{2k} a_\ell \leq S$.
 
 \medskip 
 
 To obtain a contradiction, it suffices to prove the following claim: \[a_\ell > \frac S{2k} \qquad \text{for every $\ell$.} \tag{$\ast$}\]  
 Indeed, summing over $\ell = 1,\ldots, 2k$ would contradict $\sum_{\ell=1}^{2k} a_\ell \leq S$. 
 
  \medskip 
 
\noindent \textit{Proof of $(\ast)$}. If $P_\ell$ is the one-edge path $p$--$q$, then $\{p,q\}$ is bad and the claim follows from the fact that $\min\{x_p,x_q\} + \max\{x_p,x_q\} = x_p + x_q$. Suppose now that $P_\ell$ is not the one-edge path and let $\{i,j\}$ be a bad edge in $P_\ell$.  
  If both $i$  and $j$ are internal vertices of $P_\ell$, then
 \[
 a_\ell\geq x_i+x_j
 >\max\{x_p,x_q\}+\frac{S}{2k}
 \geq\frac{S}{2k}.
 \]
 Otherwise, exactly one endpoint of $\{i,j\}$ lies in $\{p,q\}$. Say $i \in \{p,q\}$ and $j$ is internal. Since $x_i \leq \max\{x_p,x_q\}$ and $\{i,j\}$ is bad, we have 
 \[
 	a_\ell\geq x_j  > \max\{x_p,x_q\} +\frac{S}{2k}-x_i  \geq\frac{S}{2k}.
 \]
 This proves the claim and yields the required contradiction. 
\end{proof}

We also need the following elementary consequence of Helly's theorem for intervals in $\mathbb R$.
\begin{lemma}\label{lemma:Helly}
	Let $y_1, \ldots, y_r \in \mathbb R$ and $w_1, \ldots, w_r > 0$. Then there exists $t \in \mathbb R$ such that \[\max_i w_i |y_i + t| \leq \max_{p,q} \frac{w_pw_q}{w_p + w_q} | y_p - y_q|.\]
\end{lemma}
\begin{proof}
	Put $ R := \max_{p,q} \frac{w_pw_q}{w_p + w_q} | y_p - y_q|$. The condition $w_i|y_i + t| \leq R$ is equivalent to \[t \in I_i := \big[-y_i -  R/{w_i}, -y_i +  R/ {w_i}\big].\]
	Thus we need the intervals to have a nonempty intersection $\bigcap_{i} I_i \neq \varnothing.$  
	
	Intervals on the real line have the Helly property: pairwise intersection implies global intersection. It thus suffices to show that $I_p \cap I_q \neq \varnothing$ for each pair $p, q$. Indeed, $I_p$ and $I_q$ intersect precisely when 
	\[
	|y_p - y_q| \leq \frac R{w_p} + \frac R{w_q} \iff R \geq \frac{w_p w_q}{w_p + w_q} | y_p - y_q|.
	\]
	The latter statement clearly holds by definition of $R$. 
\end{proof}

By \eqref{eq:equivalentCL}, it remains to prove the following statement. 
\begin{lemma}
	Let $n \geq 2$, let $b \in \mathcal O \setminus \{0\}$ with $\|b\| \leq n$ and let $\delta \in B_b(\mathcal O)$ with $\|\delta\| \leq n^2$. Then there exists a solution $a' \in \mathcal O$ of $B_b(a') = \delta$ satisfying \[
	\|a' \| \ll n^{1 + r/(2k)}.
	\]
\end{lemma}
\begin{proof}
	Let $a \in \mathcal O$ be any solution of $B_b(a) = \delta$. By Lemma~\ref{lemma:kernel}, we can modify $a$ to find a suitable $a'$ by adding an integer multiple of $b$. 
	
	\medskip 
	
	For each $i \in \mathbb Z / r\mathbb Z$, define $x_i = \log_n (n / |\sigma_i(b)|).$ Since $\|b\| \leq n$, we have $|\sigma_i(b)| \leq n$, so $x_i \geq 0$. Also $|\sigma_i(b)| = n^{1-x_i}$. 
	
	Now use the algebraic norm $N_{K/\mathbb Q}(b) = \prod_i \sigma_i(b).$ Since $b \in \mathcal O \setminus \{0\}$, we have $|N_{K/\mathbb Q}(b)| \in \mathbb Z_{\geq 1}$. Hence\[
	1 \leq \prod_i |\sigma_i(b)| = n^{r - \sum_i x_i} \implies \sum_i x_i \leq r.
	\] 
	
	Next we translate the equation $B_b(a) = \delta$ into estimates along the edges of $C_r^k$. As in the proof of Lemma~\ref{lemma:boundDelta}, for $i \in \mathbb Z/r \mathbb Z$, $1 \leq s \leq k$, the $\sigma_i$-embedding of the $s$-component of $B_b(a) = \delta$ is \[
	\sigma_i(a) \sigma_{i+s}(b) - \sigma_i(b) \sigma_{i+s}(a) = \sigma_i(\delta_s). 
	\] 
	Dividing by $\sigma_i(b)\sigma_{i+s}(b)$ and using $\|\delta\| \leq n^2$ and $|\sigma_i(b)| = n^{1-x_i}$ yields \begin{equation}
		\left| \frac{\sigma_i(a)}{\sigma_i(b)} - \frac{\sigma_{i+s}(a)}{\sigma_{i+s}(b)} \right| = \frac{|\sigma_i(\delta_s)|}{|\sigma_i(b)||\sigma_{i+s}(b)|} \leq \frac{n^2}{n^{1-x_i}\cdot n^{1-x_{i+s}}} = n^{x_i + x_{i+s}} \label{eq:boundOnEdges}
	\end{equation}
	for every edge $\{i,i+s\}$ of $C_r^k$. 
	
	Since $\sum_i x_i \leq r$, summing the bound \eqref{eq:boundOnEdges} for all edges $\{i,j\}$ in the path given by Lemma~\ref{lemma:weightedPathLemma} yields  
	\begin{equation}
		\left| \frac{\sigma_p(a)} {\sigma_{p}(b)} - \frac{\sigma_q(a)}{\sigma_{q}(b)} \right| \leq (r-1) \, n^{\max\{x_p,x_q\} + r/(2k)}  \qquad \text{for all vertices } p,q \in \mathbb Z / r \mathbb Z. \label{eq:boundOnVerticesRatio}
	\end{equation}
	
	\medskip 
	
	Applying Lemma~\ref{lemma:Helly} with \(
	y_i = {\sigma_i(a)}/{\sigma_i(b)}\) and $w_i = |\sigma_i(b)|$ gives $t \in \mathbb R$ such that \[
	 \max_i \left| \sigma_i(a) +t \sigma_i(b) \right| \leq \max_{p,q} \frac{|\sigma_p(b)||\sigma_q(b)|}{|\sigma_p(b)| + |\sigma_q(b)|} \left|  \frac{\sigma_p(a)} {\sigma_{p}(b)} - \frac{\sigma_q(a)}{\sigma_{q}(b)}  \right|.
	\]
	We bound the right-hand side. First, use  \[\frac{|\sigma_p(b)||\sigma_q(b)|}{|\sigma_p(b)| + |\sigma_q(b)|} \leq \min \{|\sigma_p(b)|, |\sigma_q(b)|\} = \min \{ n^{1-x_p}, n^{1-x_q} \} = n^{1-\max \{x_p,x_q\}}. \]
	Together with \eqref{eq:boundOnVerticesRatio}, this gives 
	\[
	\frac{|\sigma_p(b)||\sigma_q(b)|}{|\sigma_p(b)| + |\sigma_q(b)|} \left|  \frac{\sigma_p(a)} {\sigma_{p}(b)} - \frac{\sigma_q(a)}{\sigma_{q}(b)}  \right| \ll n^{1 - \max\{x_p,x_q\}} \cdot n^{\max\{x_p,x_q\} + r/(2k)} = n^{1 + r/(2k)}.
	\]
	
	\medskip 
	
	Therefore, the affine line $a + \mathbb R b$ contains a point $a+tb$ of size $\ll n^{1 + r/(2k)}$. Replacing $t$ by an integer $m \in \mathbb Z$ satisfying $|m-t| \leq 1$ gives the required element $a' = a + m b \in \mathcal O$ also of size $\ll n^{1 + r/(2k)}$. Indeed, $B_b(a') = B_b(a) + m B_b(b) = \delta$ and for every $i$,
	\begin{align*}
		|\sigma_i(a')| = |\sigma_i(a) + m \sigma_i(b)| \leq |\sigma_i(a) + t \sigma_i(b) | + |m-t| |\sigma_i(b)| \ll n^{1 + r/(2k)} + 1 \cdot n, 
	\end{align*}
	which is $\ll n^{1 + r/(2k)}$. 
	Taking the maximum over all embeddings gives $\|a'\| \ll n^{1 + r/(2k)}$. 
\end{proof}

\section{Classifying conjugator length in nilpotent groups is ``hard''  (Theorem~\ref{thm:unboundedpartialquotients})}\label{sec:hard} 

\noindent In this section, we prove that characterizing conjugator length in 2-step nilpotent groups is at least as hard as a partial resolution of the following long-standing classical conjecture in Diophantine approximation. 

\begin{conjecture}[Khinchin \cite{Khinchin1949ContinuedFractions}]\label{conjecture:boundedPartialQuotients}
	Every real algebraic number $\theta$ of degree $\geq 3$ has unbounded partial quotients in its continued fraction expansion.
\end{conjecture}

Recall that if $\theta \in \mathbb R \setminus \mathbb Q$, its regular continued fraction expansion is \[\theta = [a_0;a_1,a_2,\ldots] := a_0+\cfrac{1}{a_1+\cfrac{1}{a_2+\ddots}},\] where \(a_0 \in \mathbb Z, \, a_1,a_2,\ldots \in \mathbb Z_{\geq 1}.\) This expansion is unique. The number $\theta$ is said to have \emph{bounded partial quotients} if $\sup_{i \geq 1} a_i < \infty$. 

The conjecture is open in the strongest possible sense: for no specific real algebraic number of degree $\geq 3$ is it known
whether its partial quotients are bounded or unbounded. For discussions of this open problem and related results, see \cite[\S4]{Shallit1992BoundedPartialQuotients}, \cite{AdamczewskiBugeaud2005ComplexityII}, \cite{AdamczewskiBugeaudDavison2006CFTrans}, \cite[Conj.\ 1.1]{BosmaGruenewald2012Complex}. 


\medskip 

Let $\theta$ be a real algebraic number of degree 3. Define the 2-step nilpotent rational Lie algebra $\mathfrak g_\theta$ as the $8$-dimensional $\mathbb Q$-Lie ideal 
\[
\mathfrak g_\theta := \big( \mathbb Q + \mathbb Q \theta \big) X \oplus \mathbb Q(\theta) Y \oplus \mathbb Q(\theta) Z \subset \mathfrak h \big( \mathbb Q(\theta) \big),
\]
where $\mathfrak h\big( \mathbb Q(\theta) \big)$ is the 9-dimensional $\mathbb Q$-Lie algebra defined in \S\ref{sec:noQI}. Equivalently, its Lie bracket is  
\[
\big[(u,x,z), (u',x',z')\big] = (0,0,ux'-u'x), \quad u,u' \in \mathbb Q + \mathbb Q\theta, \quad x,x' \in \mathbb Q(\theta), \quad z,z' \in \mathbb Q(\theta).
\]
  \setcounter{letterthmbody}{3}  
\begin{letterthmbody}\label{thmbody:unboundedpartialquotients} Let $\theta$ be a real algebraic number of degree $3$. Then
	$\mathfrak g_\theta$ satisfies: 
	\begin{enumerate}
\item $n^3 \preceq \CL_{\mathfrak g_{\theta}}(n) \preceq n^{3 + \varepsilon}$ for every $\varepsilon > 0$, 
\item $n^3 \prec \CL_{\mathfrak g_{\theta}}(n)$ if and only if a real conjugate of $\theta$ has unbounded partial quotients. 
\end{enumerate}
\end{letterthmbody}

To prove Theorem~\ref{thmbody:unboundedpartialquotients}, we first explain why we can reduce the proof to algebraic \emph{integers}. First, the construction of $\mathfrak g_\theta$ is unchanged under multiplication by a nonzero rational number. If $q \in \mathbb Q^{\times}$, then $\mathfrak g_{q \theta} = \mathfrak g_{\theta}$. Second, boundedness of the partial quotients of a real irrational number is invariant under multiplication by a nonzero rational number. Indeed, an irrational real number $\eta \in \mathbb R \setminus \mathbb Q$ has bounded partial quotients if and only if $\eta$ is \emph{badly approximable}, that is,  \begin{equation*}
	|P - Q \eta| \gg_{\eta} \frac 1 {|Q|} \qquad \text{for every } P,Q \in \mathbb Z,\, Q \neq 0, 
\end{equation*}
see \cite[Thm.\ 1.9]{Bugeaud2004Approximation}. This property is preserved under multiplication by a nonzero rational number. 

Choose a positive integer $A$ such that $\beta:=A\theta$ is an algebraic integer. Then $\mathfrak g_\theta=\mathfrak g_\beta$. Moreover, for every real conjugate $\theta_i$ of $\theta$, the corresponding conjugate $\beta_i=A\theta_i$ has bounded partial quotients if and only if $\theta_i$ does. Consequently, it suffices to prove Theorem~\ref{thmbody:unboundedpartialquotients} when $\theta$ is an algebraic integer. 



\subsection{Construction of $ L_\theta \subset \mathfrak g_\theta$} For the remainder of this section, fix a real algebraic \emph{integer} $\theta$ of degree 3. Let $m_\theta(T) = T^3 + a_2 T^2 + a_1 T + a_0$ be its minimal polynomial.

Let \(\sigma_1,\sigma_2,\sigma_3\colon\mathbb Q(\theta)\hookrightarrow\mathbb C\) be the distinct embeddings, indexed so that \(\sigma_1(\theta)=\theta\), and write \(\theta_i:=\sigma_i(\theta)\). Then \(\theta_1=\theta,\theta_2,\theta_3\) are the three conjugates of \(\theta\). For $x \in \mathbb Q(\theta)$, define the Archimedean supremum norm  $
\| x \| = \max_i |\sigma_i(x) |.
$  We work in the Lie ring
 \[
 L_\theta := (\mathbb Z + \mathbb Z \theta) \oplus \mathbb Z[\theta] \oplus \mathbb Z[\theta] \subset \mathfrak g_{\theta}
 \]
 and use the Guivarc'h length
 \[
 \big\| (u,x,z) \big\|_{\mathrm{Gv}} = \max \{ \|u\|, \|x\|, \|z\|^{1/2}\}.
 \]

\subsection{Proof of Theorem~\ref{thmbody:unboundedpartialquotients} for $L_\theta$} By \S\ref{sec:reductiontoLierings}, it suffices to prove Theorem~\ref{thmbody:unboundedpartialquotients} for $L_\theta$. 

For $x \in \mathbb Q(\theta)$, put $s(x) = \min_i |\sigma_i(x)|$. Define the \emph{small-divisor function} $s_\theta \colon \mathbb N \to \mathbb R_{>0}$ by 
\[
s_\theta(n) = \min \big\{ s(\alpha) \:\big|\: 0 \neq \alpha \in \mathbb Z + \mathbb Z \theta, \, \| \alpha \| \leq n \big\}.
\]
Observe that $s_\theta$ is well-defined: the set over which we take the minimum is finite for every $n\geq 1$ and it always contains $1 \in \mathbb Z + \mathbb Z\theta$. 

In the following two lemmas, we prove Theorem~\ref{thmbody:unboundedpartialquotients} for $L_\theta$ by first proving the necessary estimates on $n^2 / s_\theta(n)$ and then showing that $\CL_{L_\theta}(n) \simeq n^2 / s_\theta(n)$. 
\begin{lemma}\label{lemma:estimates-on-n2-mtheta} Let $\theta$ be a real cubic algebraic integer. Then:
	\begin{enumerate}
	\item \(n^3 \preceq \cfrac{n^2}{s_\theta(n)} \preceq n^{3 + \varepsilon}\) for every $\varepsilon>0$. 
	\label{item:estimates-on-mtheta}
	\item $n^3 \prec \cfrac{n^2}{ s_\theta(n)} $ if and only if a real conjugate of $\theta$ has unbounded partial quotients. \label{item:unbounded-partial-quotients}
	\end{enumerate}
\end{lemma}
\begin{proof}
	
\noindent \eqref{item:estimates-on-mtheta}. It suffices to show \begin{equation}
	n^{-1-\varepsilon} \ll_{\varepsilon} s_\theta(n) \ll n^{-1}. \label{eq:estimates-on-n2-mtheta}
	\end{equation}
	
	First we prove $s_\theta(n) \ll n^{-1}$. Fix $n \geq 1$. By Dirichlet's approximation theorem, there exist integers $P,Q$ with $1 \leq Q \leq n$ such that $|P - Q\theta| \leq 1/n$. Put $\alpha = P-Q \theta$.  
	We must check that $\|\alpha\| \ll n$. For the distinguished embedding, this is clear: $|\sigma_1(\alpha)| = |P - Q \theta| \leq 1/n$. For the other embeddings ($i = 2,3$), we have
	\[
	|\sigma_i(\alpha)| = |P - Q\theta_i| \leq |P-Q\theta| + Q |\theta - \theta_i| \ll n. 
	\]
	Since $s(\alpha) \leq |\sigma_1(\alpha)| = |P - Q \theta| \leq 1/n$, this proves  $s_\theta(Cn) \ll n^{-1}$ for some constant $C \geq 1$ depending only on $\theta$. After increasing constants and using that $s_\theta$ is nonincreasing, we obtain $s_\theta(n) \ll n^{-1}$. 
	
	\medskip 
	
Fix $\varepsilon > 0$. Now we prove $s_\theta(n) \gg_\varepsilon n^{-1-\varepsilon}$. 	Let $0 \neq \alpha = P - Q\theta \in \mathbb Z + \mathbb Z\theta$ with $\|\alpha\| \leq n$. First note that $|Q| \ll n$. Indeed, 
	\[
	|Q| = \frac{|\alpha - \sigma_2(\alpha)|}{|\theta - \theta_2|} \leq \frac{|\alpha| + |\sigma_2(\alpha)|}{|\theta - \theta_2|} \ll \|\alpha\| \leq n.
	\]
	It remains to show that $|\sigma_i(\alpha)| \gg_\varepsilon n^{-1-\varepsilon}$ for a fixed embedding $\sigma_i$. If $Q = 0$, then $\alpha = P \neq 0$, so $|\sigma_i(\alpha)| = |P| \geq 1 \geq n^{-1-\varepsilon}.$ 
	Assume $Q \neq 0$. If $\theta_i$ is real, then $\theta_i$ is a real algebraic irrational of degree 3. Roth's theorem (see \cite[Ch.\ V, Thm.\ 2A]{Schmidt1980DiophantineApproximation}, \cite{Rot55}) gives 
	\[
	\left|\theta_i - \frac PQ\right| \gg_{\varepsilon} \frac 1 {|Q|^{2+\varepsilon}}.
	\]
	Therefore, \[
	|\sigma_i(\alpha)| = |Q| \cdot \left|\theta_i - \frac PQ\right| \gg_\varepsilon |Q|^{-1-\varepsilon} \gg_\varepsilon n^{-1-\varepsilon}.
	\]
If $\theta_i$ is non-real, then similarly $|\sigma_i(\alpha)| \gg |Q| \cdot |\mathfrak{Im} (\theta_i)| \gg  1 \geq n^{-1-\varepsilon}$ since $P/Q \in \mathbb R$. This completes the proof of \eqref{item:estimates-on-mtheta}.

\medskip 

\noindent \eqref{item:unbounded-partial-quotients}.	By \eqref{eq:estimates-on-n2-mtheta}, it suffices to show that $s_\theta(n) \gg n^{-1}$ if and only if every real conjugate of $\theta$ has bounded partial quotients. We use the standard result (see, for example, \cite[Thm.\ 1.9]{Bugeaud2004Approximation}) that $\theta_i \in \mathbb R\setminus \mathbb Q$ has bounded partial quotients if and only if $\theta_i$ is \emph{badly approximable}, that is, 
\begin{equation}
	|P - Q \theta_i| \gg_{\theta_i} \frac 1 {|Q|} \qquad \text{for every } P,Q \in \mathbb Z,\, Q \neq 0. \label{eq:bounded-partial-quotients}
\end{equation}

First assume that \eqref{eq:bounded-partial-quotients} holds for each real conjugate $\theta_i$.  
Let $n \geq 1$ and let $0 \neq \alpha = P - Q \theta \in \mathbb Z + \mathbb Z \theta$ with $\|\alpha \| \leq n$. The same argument as in the proof of Lemma~\ref{lemma:estimates-on-n2-mtheta}, with Roth's theorem replaced by \eqref{eq:bounded-partial-quotients}, gives $|\sigma_i(\alpha)| \gg n^{-1}$ for each $i$. Hence, $s_\theta(n) \gg n^{-1}$.

\medskip 

For the converse direction, assume that $s_\theta(n) \gg n^{-1}$. Let $\theta_i$ be a real conjugate of $\theta$. We prove \eqref{eq:bounded-partial-quotients}. Let $P,Q \in \mathbb Z$ with $Q \neq 0$ and put $\alpha := P-Q \theta \neq 0$. 
If $|P-Q\theta_i| \geq 1$, then automatically $|P-Q\theta_i|\geq 1/|Q|$. So assume $|P-Q\theta_i| < 1$. Then $|P| \ll |Q|$, so 
\[
|\sigma_j(\alpha)| \leq |P| + |Q| |\theta_j| \ll |Q|
\]
for every embedding $\sigma_j$. Hence $\|\alpha\| \ll |Q|$. Since $s_\theta(n) \gg n^{-1}$, we deduce
\[
|P-Q\theta_i| \geq s(\alpha) \gg \frac 1{|Q|}.
\]
This shows that $\theta_i$ has bounded partial quotients, completing the proof.
\end{proof}

\begin{lemma}\label{lemma:CL-is-n2-mtheta}
	$\CL_{L_\theta}(n) \simeq \cfrac{n^2}{s_\theta(n)}$.
	\end{lemma}
\begin{proof} First we deduce the conjugacy equation. Let $g = (\alpha, y,z), h = (\alpha,y,z + \delta) \in L_\theta$. Then $w = (\beta,x,0) \in L_\theta$ conjugates $g$ to $h$ if and only if
	\begin{equation}
 g + [w,g] = h \iff y \beta - \alpha x = \delta \label{eq:conjugacy-equation}.
	\end{equation}
	
	\medskip 
	
\noindent \textit{Upper bound} $\CL_{L_\theta}(n) \preceq n^2 / s_\theta(n)$\textit{.} Let $g,h \in L_\theta$ be conjugate with $\| g\|_{\mathrm{Gv}}, \| h\|_{\mathrm{Gv}} \leq n$. As above, write $g = (\alpha, y,z), h = (\alpha,y,z + \delta)$. Then $\|\alpha \| \leq n$,  $\|y\| \leq n$ and $\|\delta \| \ll n^2$. 
Since $g$ is conjugate to $h$, equation \eqref{eq:conjugacy-equation} has at least one solution. We must modify this solution to find $(\beta,x)$ such that 
\[
\| \beta \|, \| x \| \ll \frac{n^2}{s_\theta(n)}.
\]

\noindent \textit{Case 1: $\alpha = 0$.} Then \eqref{eq:conjugacy-equation} becomes \[y\beta   = \delta.\]
Write $\beta = a + b \theta$. The norm $\|\cdot\|$ is bilipschitz equivalent to the usual supremum norm on $\mathbb Q(\theta)$ with respect to the basis $1,\theta,\theta^2$. Since $ y \beta = \delta$ is a consistent $3\times 2$ integer linear system of equations, with coefficients $\ll n$ and right-hand side $\ll n^2$, the Borosh--Flahive--Rubin--Treybig bound \cite{BFRT89} (as quoted in \cite[Thm.~4]{BR26TwoStep}) gives a solution $\beta$ of size $\|\beta\|\ll n^3$, after deleting dependent rows if necessary. By Lemma~\ref{lemma:estimates-on-n2-mtheta}, $\|\beta\| \ll n^3  \ll n^2 / s_\theta(n)$, so we can take $(\beta,0)$ as the required short solution to \eqref{eq:conjugacy-equation}.

\medskip  

\noindent \textit{Case 2: $\alpha \neq 0$.} Write $\alpha = P - Q \theta$ and define the field norm
\[
N(\alpha) := N_{\mathbb Q(\theta)/\mathbb Q}(\alpha) := \alpha \cdot \sigma_2(\alpha) \cdot \sigma_3(\alpha) = \prod_{i=1}^3 (P - Q\theta_i) = P^3 + a_2 P^2 Q + a_1 PQ^2 + a_0 Q^3 \in \mathbb Z_{\neq 0},
\] 
where the $a_i$ are the coefficients of the minimal polynomial $m_\theta(T)$. 

Notice that $N(\alpha) \in \alpha \mathbb Z[\theta]$. Indeed, one computes that $N(\alpha) = \alpha \alpha^{\#}$ where 
\begin{equation}
 \alpha^{\#} := (P^2 + a_2 PQ + a_1 Q^2) + (PQ + a_2 Q^2) \theta + Q^2 \theta^2 \in \mathbb Z[\theta]. \label{eq:adjugate2}
\end{equation}
Put $M := |N(\alpha)|$. Then $M = \prod_i |\sigma_i(\alpha)| \leq \|\alpha\|^3 \leq n^3. $

Let $(\beta_0,x_0)$ be any solution of \eqref{eq:conjugacy-equation}. Write $\beta_0 = a_0 + b_0 \theta$. Reducing coefficients modulo~$M$ yields $\beta = a + b \theta$ with $ |a|,|b|  \leq M$. Then $\| \beta \| \ll M$. Moreover, $\beta_0 - \beta \in M \mathbb Z[\theta]$ and since $M = \pm N(\alpha) \in \alpha \mathbb Z[\theta]$, we have $y(\beta_0-\beta) \in \alpha \mathbb Z[\theta]$. Since $(\beta_0,x_0)$ solves \eqref{eq:conjugacy-equation}, it follows that 
\[
 y \beta - \delta = (y \beta_0 - \delta) - y(\beta_0-\beta) \in \alpha \mathbb Z[\theta].
\]
Choose $x \in \mathbb Z[\theta]$ such that \(
\alpha x =  y \beta - \delta.
\) Then $(\beta,x)$ solves \eqref{eq:conjugacy-equation}. 
Since $|a|,|b| \leq M$, it follows from $M \leq n^3$  that $\|\beta\| \ll M \ll n^2 / s_\theta(n)$. It remains to bound $\|x\|$ in terms of $n^2 / s_\theta(n)$. 

For each embedding $\sigma_i$, we have \(
|\sigma_i(x)| = | \sigma_i(y \beta - \delta)| / |\sigma_i(\alpha)|. 
\) Hence
\[
\| x \| \leq \frac{\| y \beta - \delta\|}{s(\alpha)}.
\]
Using \(\|xy\| \leq \|x\| \|y\|\), we deduce that \(\|y \beta - \delta\| \leq  \|y\| \|\beta\| + \|\delta\| \ll Mn + n^2,\) hence\[
\|x\| \ll \frac{Mn}{s(\alpha)} + \frac{n^2}{s(\alpha)}.
\]
For the first term, choose an embedding $\sigma_{i_0}$ such that $s(\alpha) = |\sigma_{i_0}(\alpha)|$.  Then 
\[
\frac{Mn}{s(\alpha)} =  n \prod_{i \neq i_0} |\sigma_i(\alpha)|  \leq n \|\alpha \|^2 \leq n^3 \ll \frac{n^2}{s_\theta(n)}.
\]
For the second term, $n^2 / s(\alpha) \leq n^2 / s_\theta(n)$. Together with $\|\beta\| \ll n^2 / s_\theta(n)$, this proves the upper bound.

\medskip 

\noindent \textit{Lower bound} $n^2 / s_\theta(n) \preceq \CL_{L_\theta}(n)$. 
Let $n \geq 1$.  
Let $0 \neq \alpha \in \mathbb Z + \mathbb Z \theta$ with $\|\alpha\| \leq n$ and $s(\alpha) = s_\theta(n)$. 
Let $\alpha^{\#} \in \mathbb Z[\theta]$ with $N(\alpha) = \alpha \alpha^{\#}$ be as in \eqref{eq:adjugate2} above. 
Put $M := |N(\alpha)|.$

First, we construct an element $\eta \in \mathbb Z[\theta]$ such that \begin{equation*}
\|\eta \| \asymp \frac{n^2}{s_\theta(n)}, \qquad \| \alpha \eta \| \leq n^2. 
\end{equation*}
Set $T := \lfloor n^2 / M \rfloor$. 
Since $s_\theta(n) \leq 1$, we have $M \leq s(\alpha) \|\alpha\|^2 \leq n^2$. Hence $T \asymp  n^2 / M$ for all $n \geq 1$. Define $\eta := T \alpha^{\#} \in \mathbb Z[\theta]$. Then 
\[
\| \eta \| = T \|\alpha^\# \| = T \cdot \frac{M}{s(\alpha)} \asymp \frac{n^2}{s(\alpha)},
\]
and $\alpha \eta = T \alpha \alpha^{\#}$, so \(
\| \alpha \eta\| = T M \ll n^2. 
\)

\medskip 

Now we use $\alpha$ and $\eta$ to prove the lower bound. Set $\delta := - \alpha \eta$. Consider $g = (\alpha,0,0)$ and $h = (\alpha, 0, \delta)$. Since $\|\alpha\| \leq n$ and $\| \alpha \eta \| \leq n^2$, we have $\|g\|_{\mathrm{Gv}}, \|h\|_{\mathrm{Gv}} \leq n$. Moreover, $w_0 = (0,\eta,0)$ conjugates $g$ to $h$, since $(0,\eta)$ solves the conjugacy equation \eqref{eq:conjugacy-equation} with $y = 0$.  
Let $w = (\beta,x,z)$ be any conjugator from $g$ to $h$. Then the conjugacy equation yields $-\alpha x = \delta$. Since $\alpha \neq 0$ and $-\alpha \eta = \delta$, this forces $x = \eta$. Therefore every conjugator from $g$ to $h$ has length at least 
\[
\|(\beta,\eta,z)\|_{\mathrm{Gv}}  \geq \| \eta \| \asymp \frac{n^2}{s_\theta(n)}.
\]
This completes the proof.
\end{proof}

\section*{Acknowledgments and AI tool disclosure} 
\noindent GPT-5.5 Pro explored different families of graph Lie algebras for the proof of Theorem~\ref{thm:denseCLspec} and suggested the right one. GPT-5.5 Pro also pointed us to the connection with Khinchin's conjecture on unbounded partial quotients for Theorem~\ref{thm:unboundedpartialquotients}. GPT-5.6 Sol and M365 Copilot were used for copy-editing, grammar checks and proofreading. All proofs were human-written and verified by the authors, who take full responsibility for the content and correctness of the manuscript.

\bibliographystyle{alpha}

\end{document}